\documentclass[11pt,reqno,oneside]{amsart}

\usepackage{graphicx,subfigure,threeparttable}
\usepackage{amssymb}
\usepackage{epstopdf}
\graphicspath{{direct_result_figure/}}
\usepackage[a4paper, total={6in, 9in}]{geometry}

\usepackage{algorithm}
\usepackage{algpseudocode}
\usepackage{amsmath}

\usepackage{booktabs} 
\usepackage{subfig} 

\usepackage{amsfonts}

\usepackage{amsmath,amsthm,mathrsfs,amssymb,cite}
\usepackage[usenames]{color}

\newtheorem{thm}{Theorem}[section]

\newtheorem{lem}{Lemma}[section]

\theoremstyle{definition}

\theoremstyle{remark}

\newtheorem{rem}{Remark}[section]
\numberwithin{equation}{section}

\numberwithin{equation}{section}

\newcounter{saveeqn}

\newcommand{\beq}{\begin{equation}}
\newcommand{\eeq}{\end{equation}}

\title{A novel direct imaging method for passive inverse obstacle scattering problem}

\author{Yunwen Yin}
\address{School of Mathematics, Southeast University, Nanjing, China}
\email{yunwenyin@seu.edu.cn}

\author{Liang Yan}
\address{School of Mathematics, Southeast University, Nanjing, China}
\email{yanliang@seu.edu.cn}

\date{} 

\begin{document}
\maketitle

\begin{abstract}

This paper investigates the inverse scattering problem of recovering a sound-soft obstacle using passive measurements taken from randomly distributed point sources. The randomness introduced by these sources poses significant challenges, leading to the failure of classical direct sampling methods that rely on scattered field measurements. To address this issue, we introduce the Doubly Cross-Correlating Method (DCM), a novel direct imaging scheme that consists of two major steps. Initially, DCM creates a cross-correlation between two passive measurements.  This specially designed cross-correlation effectively handles the uncontrollability of incident sources and connects to the active scattering model via the Helmholtz-Kirchhoff identity. Subsequently, this cross-correlation is used to create a correlation-based imaging function that can qualitatively identify the obstacle. The stability and resolution of DCM are theoretically analyzed.  Extensive numerical examples, including scenarios with two closely positioned obstacles and multiscale obstacles, demonstrate that DCM is computationally efficient, stable, and fast.

\noindent{\bf Keywords:}~~ inverse obstacle scattering problem; passive imaging; direct imaging scheme; cross-correlation.

\medskip
\end{abstract}

\section{Introduction}
Inverse scattering problems arise in many practical applications, for example, radar imaging \cite{B1}, nondestructive testing \cite{ABF} and medical imaging \cite{Kuchment}. One important class of such inverse problems is the inverse obstacle scattering problem that focuses on determining the location and shape of the object. This problem is highly non-linear and severely ill-posed, which means that its solution may not exist or may not be unique, and more importantly may not be stable with respect to the noise. As a result, developing effective and stable numerical strategies is challenge but necessary. Existing reconstruction approaches can be broadly classified into three types: iterative methods, decomposition methods, and sampling-type methods.  Iterative methods, such as Newton-type iterative method \cite{Kress,KR} and recursive linearization method \cite{Bao_Lin_Mefire_1,Borges_Greengard}, need to repeatedly recall the forward solver during their inversion processes, and thus they are typically faced with the curse of the high computational costs. Decomposition methods reformulate the inverse problem as a non-linear optimization problem containing two parts: the first part is to recover the scattered field from the farfield pattern, and the second part uses the boundary condition and the recovered scattered field to determine the boundary of the scatterer. It is commendable that decomposition methods are capable of not only well avoiding a sequence of forward model evaluations, but also providing satisfactory reconstructions. We refer the readers to \cite{GMZ,KK2,KK1,S1} for the relevant results of decomposition methods. Rather,  sampling-type methods, also known as qualitative methods, such as the factorization method\cite{kirsch2007factorization,KL1,Yangjiaqing1,zhang_zhang1}, the direct sampling method \cite{ji_xia1,LZ,lxd1,liu_meng_zhang1}  and the orthogonality sampling method \cite{potthast1}, solve the inverse obstacle scattering problem from a completely different perspective. They are devoted to designing a specific indicator function to qualitatively visualize the unknown obstacle.

The aforementioned reconstruction approaches have had enormous success with the inverse obstacle scattering problem. However, they are based on active imaging, which means that the incident field described by a time harmonic plane wave or point source is under control. Unfortunately, in many practical scenarios, the incident wave is random and unavailable, so the receivers passively record the response of the unknown impenetrable obstacle. The problem of detecting an obstacle using such passive measurements is known as passive imaging. In general, fewer studies on passive imaging have been conducted than on active imaging. In \cite{GHM1}, a novel linear sampling method was proposed to solve inverse obstacle scattering with random sources using the Helmholtz-Kirchhoff identity and cross-correlation computed at two measurements. This work greatly contributes to the study of inverse scattering problems. Recently, the concept has been extended to passive imaging with small random scatterers \cite{GHM2}. We refer to \cite{GHM1,GHM2} for more detailed discussions of this method. However, because the underlying model error associated with the passive model and the approximate active model is ignored, the reconstruction will be affected accordingly.  In order to overcome this, a Bayesian model error method was presented in \cite{Yin_Yan_1} with an online estimating of the model error. We strongly recommend \cite{GP1} for a comprehensive understanding of various interesting and meaningful passive inverse scattering problems. Furthermore, we would like to point out that the introduced passive problem can be solved in terms of the co-inversion of the obstacle and its excitation sources\cite{CG1,CGLZ1,Zhang_Chang_Guo1,ZGWC1}.

In this work, we introduce a new direct sampling method called Doubly Cross-Correlating Method (DCM) for the passive inverse obstacle scattering with randomly distributed incident sources. The inherent randomness of these sources makes the use of scattered field data for direct detection impractical.  To overcome this challenge, DCM constructs cross-correlation data from passive measurements. By leveraging the Helmholtz-Kirchhoff identity, this cross-correlation can be tightly connected to the active scattering model. Consequently, the methodologies developed for active inverse scattering problems can be utilized, leading to the establishment of a correlation-based imaging functional inspired by by the reverse time migration method \cite{Chen_zhi_ming2,Chen_zhi_ming4,Chen_zhi_ming1,Chen_zhi_ming6,Chen_zhi_ming5,Chen_zhi_ming3,Chen_zhi_ming7,li_jian_liang1,li_wu_yangjiaqing1,li_yangjiaqing1}. Specifically, we first back-propagate the well-constructed cross-correlation into the background medium and then formulate an imaging functional using the imaginary part of the cross-correlation between the fundamental solution and the back-propagated field. To the best of our knowledge, the proposed DCM represents the first approach for reconstructing impenetrable obstacles in passive imaging scenarios from a direct sampling perspective. Unlike the method presented by \cite{GHM1}, DCM avoids solving the linear integral system. The main features and novelties of our work are summarized as follows:

\vspace{0.25cm}
\begin{itemize}
\item We propose a novel direct imaging method to recover the location and shape of sound-soft obstacles in passive scenarios. This method combines cross-correlation from two passive measurements, the Helmholtz-Kirchhoff identity, and the principles of reverse time migration.
\item Instead of directly using passive measurements, we employ a specially designed cross-correlation that effectively addresses the challenges posed by the randomness and uncontrollability of incident point sources.
\item Theoretical analyses of resolution and stability are rigorously proved. Extensive numerical examples demonstrate the method's effectiveness and efficiency. Our approach achieves satisfactory reconstructions even for complex scenarios such as two closely positioned obstacles and multiscale obstacles.
\item The proposed direct imaging method involves only matrix multiplication, making it computationally fast and simple to implement. Additionally, our method shows significant stability in the presence of noise.
\end{itemize}
\vspace{0.25cm}

The rest of the paper is organized as follows. In Section 2, the passive scattering model with randomly distributed point sources is introduced. In Section 3, the DCM for the passive inverse scattering problem is presented, along with the resolution and stability analyses.   Numerical experiments are conducted in Section 4 to show the promising features of the proposed DCM. Finally, we make some conclusions in Section 5.

\section{Problem setup}
In this section, we shall present the mathematical descriptions of the direct and inverse scattering problem.
\subsection{Direct scattering problem}
Assume that $D\subset \mathbb{R}^{2}$ is a bounded simply connected sound-soft obstacle with a $C^{2}$-boundary $\partial{D}$. Given the time-harmonic point source located at $z\in \mathbb{R}^{2}\setminus \overline{D}$, the incident field $u^{i}$ can be described as
\begin{equation}\label{eq:inc}
    \begin{aligned}
    u^{i}(x,z):=\phi(x,z)=\frac{{\mathrm{i}}}{{4}}H_{0}^{(1)}(k|x-z|),\ \  x\in \mathbb{R}^{2}\setminus \{\overline{D}\cup \{z\}\},
    \end{aligned}
\end{equation}
where $\phi(x,z)$ denotes the outgoing Green function of Helmholtz equation in $\mathbb{R}^2$, $H_{0}^{(1)}$ is the Hankel function of the first kind of order zero, $k\in \mathbb{R}_{+}$ is the wave number and $\mathrm{i}:=\sqrt{-1}$ is the imaginary unit. The presence of the sound-soft obstacle $D$ will interrupt the propagation of the incident wave $u^{i}$, leading to the exterior scattered field $u^{s}$. The direct scattering problem is to find $u^{s}=u-u^{i}$ that satisfies the following Helmholtz system:
\begin{equation}\label{eq:helmholtz}
    \begin{aligned}
    \bigtriangleup u^{s}(x,z)+k^{2}u^{s}(x,z)=0\ \ \mbox{in}\  \mathbb{R}^2 \backslash \overline{D},
    \end{aligned}
\end{equation}
\begin{equation}\label{eq:bc}
    \begin{aligned}
    u^{s}(x,z)=-u^{i}(x,z)\ \ \mbox{on}\  \partial{D},
    \end{aligned}
\end{equation}
\begin{equation}\label{eq:rc}
    \begin{aligned}
    \lim_{r=|x|\rightarrow\infty}\sqrt{r}\Big( \frac{\partial{u^{s}(x,z)}}{\partial{r}}-\mathrm{i}ku^{s}(x,z)\Big)=0,
    \end{aligned}
\end{equation}
where $u$ is the total field and the Sommerfeld radiation condition \eqref{eq:rc} describes the outgoing nature of the scattered field $u^{s}$ and holds uniformly in all directions. The well-posedness of the forward scattering problem \eqref{eq:helmholtz}-\eqref{eq:rc} is well known (cf. \cite{DRIA,McLean1}).

\subsection{Passive inverse obstacle scattering problem}
We are now going to introduce the considered passive inverse obstacle scattering problem. Let $\partial{B}$, which is a circle of the radius $r_{B}$ and contains $D$ in its interior, be the measurement curve. Measurement points $x_{j}(j=1,\cdots,J)$ are located on $\partial{B}$. Let $z_{l}(l=1,\cdots,L)$ be the positions of incident point sources $\phi(x,z_{l})$ that are randomly distributed on $\Sigma$, where $\Sigma$ denotes the incident curve that encloses $\partial{B}$. Moreover, the random position $z_{l}$ is far away from the obstacle $D$. Then, the corresponding inverse problem is to determine the location and shape of the sound-soft obstacle $D$ from the passively collected total field $u(x_{j},z_{l})$. We refer to Fig. \ref{fig:passive_active_pic}(a) for the specific illustration of this passive inverse scattering problem. In order to more intuitively describe the difference of passive imaging and active imaging, we also present illustration of active imaging in Fig. \ref{fig:passive_active_pic}(b). In the active setup, measurement points $x_{j}$ and source positions $x_{m}$ are both controlled and located on $\partial{B}$. Clearly, the main difference of these two problems is whether excitation sources are controlled. Unfortunately, we can not place active excitation sources particularly in the practical scenario that one needs to face the limitation of safety or environment, and meanwhile passive imaging viewpoint can be naturally adopted.
\begin{figure}[t]
    \centering
    \subfigure[Passive imaging]{
    \includegraphics[width=0.3\textwidth]{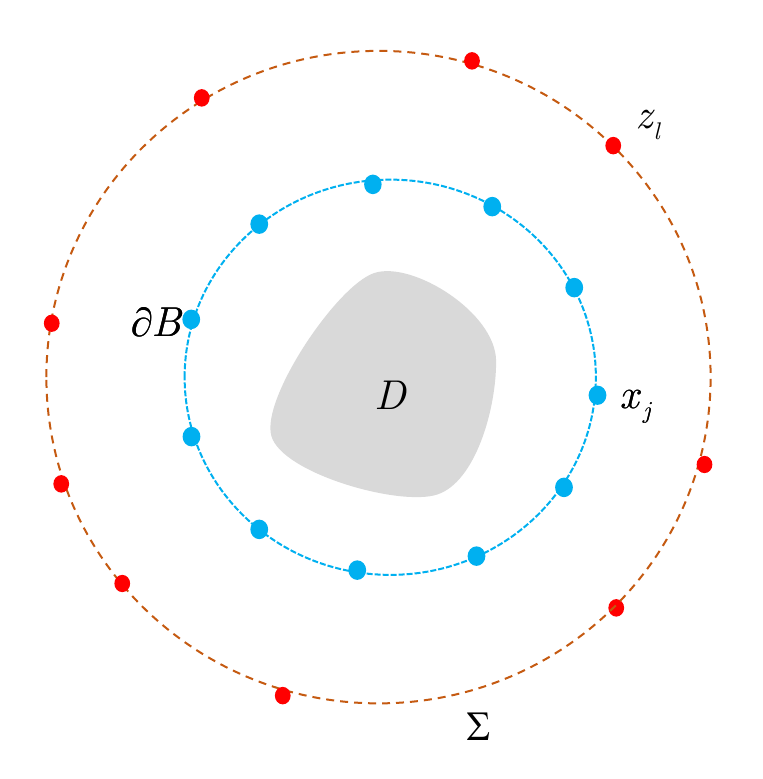}
}
    \subfigure[Active imaging]{
    \includegraphics[width=0.375\textwidth]{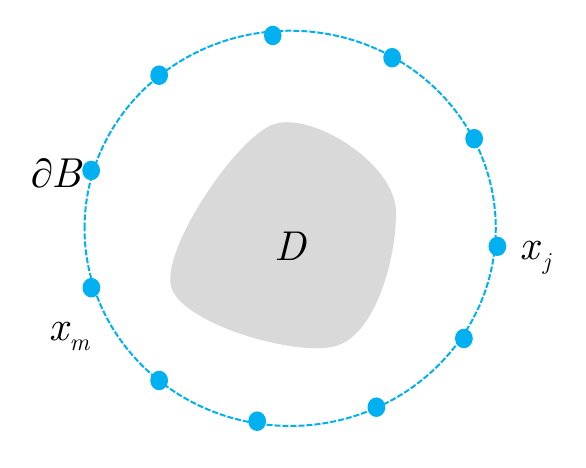}
}
    \caption{\label{fig:passive_active_pic}In the passive imaging (Left), the incident sources are random and uncontrolled, but in the active imaging (Right), the incident sources are controlled and fixed.}
\end{figure}

In this paper, we want to design a direct sampling method for solving this problem.  The direct sampling method has emerged as one of the most popular reconstruction techniques for inverse scattering problems due to its efficient computing capabilities. The fundamental idea is to develop a distinct imaging functional that can be used to qualitatively visualize the scatterer by applying its values at various probing points.  However, in our considered passive imaging setting, directly using the scattered field $u^{s}(x_{j},z_{l})$ to design the traditional direct sampling method may not be a good choice because of the randomness of the position $z_{l}$. In order to overcome this challenge, this paper proposes a novel direct imaging scheme whose specific details are presented in the following section.

\section{The doubly cross-correlating method}
In this section, we introduce our proposed Doubly Cross-Correlating Method (DCM) for addressing the passive inverse obstacle scattering problem. The DCM framework primarily involves two key steps: initially constructing a cross-correlation between two passive measurements, and subsequently utilizing this cross-correlation to design a correlation-based indicator functional. Essentially, our method employs a secondary imaging correlation derived from the primary data correlation computed at two passive receivers, thus aptly termed ``doubly cross-correlating". Additionally, we will present a thorough analysis of the resolution and stability of this method, highlighting its robustness and efficacy in practical applications.

\subsection{DCM for the inverse problem}
To describe our method, we use mathematical notations similar to works
\cite{GHM1,Yin_Yan_1}. Denote by $\Omega$ a sampling domain that contains the obstacle $D$. Before we continue, we recall $\phi(x,y)=\frac{{\mathrm{i}}}{{4}}H_{0}^{(1)}(k|x-y|)$ that is the fundamental solution of the Helmholtz equation
\begin{equation}\label{eq:back_fundamental_helmholtz}
    \begin{aligned}
    \bigtriangleup \phi(x,y)+k^{2}\phi(x,y)=-\delta_{y}(x)\ \ \mbox{in}\  \mathbb{R}^2,
    \end{aligned}
\end{equation}
and satisfies the Sommerfeld radiation condition
\begin{equation}\label{eq:back_fundamental_rc}
    \begin{aligned}
    \lim_{r=|x|\rightarrow\infty}\sqrt{r}(\frac{\partial{\phi}}{\partial{r}}-\mathrm{i}k\phi)=0,
    \end{aligned}
\end{equation}
where $\delta_{y}(\cdot)$ is the Dirac source located at $y$. In order to alleviate the difficulty caused by random sources at positions $z_{l}$, we compute the cross-correlation matrix with its entries $C_{jm}$ defined at $x_{j}$ and $x_{m}$ by
\begin{equation}\label{eq:C}
    \begin{aligned}
    C_{j m}=\frac{2 \mathrm{i} k|\Sigma|}{L} \sum_{l=1}^L \overline{u\left(x_j, z_{l}\right)} u\left(x_m, z_{l}\right)-\left[\phi\left(x_j, x_m\right)-\overline{\phi\left(x_j, x_m\right)}\right], \quad 1 \leq j, m \leq J,
    \end{aligned}
\end{equation}
where $|\Sigma|$ denotes the area of the surface $\Sigma$ on which incident sources are randomly distributed, and $\phi\left(x_j, x_m\right)$ is the measurement collected at $x_{j}$ of the point source located at $x_{m}$.

In what follows, we shall connect the cross-correlation \eqref{eq:C} to the active scattering model. To this end, let $x, y$ be far from $\Sigma$. Then, the Helmholtz-Kirchhoff identity for $\phi(x,y)$ is described as \cite{GHM1,Yin_Yan_1}
\begin{equation}\label{eq:HK_inc}
    \begin{aligned}
    \phi(x, y)-\overline{\phi(x, y)}=2 \mathrm{i} k \int_{\Sigma} \overline{\phi(x, z)} \phi(y, z) \mathrm{d} s(z).
    \end{aligned}
\end{equation}
We remark that \eqref{eq:HK_inc} holds for the case that $x$ and $y$ are far from $\Sigma$, namely for $z\in\Sigma$, $|x-z|$ and $|y-z|$ are both required to be large enough. Viewing $\phi(x,y)$ as the incident source located at $y$, the corresponding total field measured at $x$ is given by $u(x,y)=\phi(x,y)+u^{s}(x,y)$. The Helmholtz-Kirchhoff identity also holds for $u(x,y)$ \cite{GHM1,Yin_Yan_1}, that is,
\begin{equation}\label{eq:HK_total}
    \begin{aligned}
    u(x, y)-\overline{u(x, y)}=2 \mathrm{i} k \int_{\Sigma} \overline{u(x, z)} u(y, z) \mathrm{d} s(z).
    \end{aligned}
\end{equation}
The formulas \eqref{eq:HK_inc} and \eqref{eq:HK_total} immediately give rise to
\begin{equation}\label{eq:HK_scattered}
    \begin{aligned}
    u^s(x, y)-\overline{u^s(x, y)}=2 \mathrm{i} k \int_{\Sigma} \overline{u(x, z)} u(y, z) \mathrm{d} s(z)-[\phi(x, y)-\overline{\phi(x, y)}].
    \end{aligned}
\end{equation}
By replacing $x$ and $y$ by $x_{j}$ and $x_{m}$ respectively, we obtain
\begin{equation}\label{eq:HK_scattered_jm}
    \begin{aligned}
    u^s(x_{j}, x_{m})-\overline{u^s(x_{j}, x_{m})}=2 \mathrm{i} k \int_{\Sigma} \overline{u(x_{j}, z)} u(x_{m}, z) \mathrm{d} s(z)-[\phi(x_{j}, x_{m})-\overline{\phi(x_{j}, x_{m})}].
    \end{aligned}
\end{equation}
Note that $u^s(x_{j}, x_{m})$ can be regarded as the active scattered field measured at $x_{j}$ due to the illumination by an active source located at $x_{m}$. For the sake of the following use, we define
\begin{equation}\label{eq:N_jm}
    \begin{aligned}
    N^{s}_{jm}=u^s(x_{j}, x_{m})-\overline{u^s(x_{j}, x_{m})},\quad 1 \leq j, m \leq J.
    \end{aligned}
\end{equation}
Therefore, the relationship between the cross-correlation \eqref{eq:C} and the imaginary scattered field \eqref{eq:N_jm} is
\begin{equation}\label{eq:app_c_scattered}
    \begin{aligned}
    C_{jm} \approx N^{s}_{jm}.
    \end{aligned}
\end{equation}
In fact, \eqref{eq:C} is the discrete approximation form of \eqref{eq:N_jm}, and the approximation accuracy between these two parts relies on the total number $L$ and the random position $z_{l}$. See Fig. \ref{fig:passive_active_pic}(b) for the illustration of the introduced active scattering problem.

Before we proceed, let us explain why traditional direct sampling methods will fail in passive imaging.  Indeed, traditional direct sampling methods are constructed by using the scattered field $u^{s}(x_{j},z_{l})$, but in passive imaging setup only total field $u(x_{j},z_{l})$ is accessible and the scattered field can not be attainable due to the randomness of $z_{l}$. Even though one can have $u^{s}(x_{j},z_{l})$, the influence of randomness of $z_{l}$ still exists. To show this, we write the indicator functional of classical reverse time migration method by using $u^{s}(x_{j},z_{l})$, which is in the form
\begin{equation}\label{eq:passive_direct_scattered_rtm_ind}
    \begin{aligned}
    I_{RTM}(\tau)=-k^{2}\mathrm{Im}\left\{\frac{2\pi r_{B}|\Sigma|}{J}\sum\limits_{l=1}^{L}\sum\limits_{j=1}^{J}\phi(\tau,z_{l})\phi(\tau,x_{j})\overline{u^{s}(x_{j},z_{l})}\right\},
    \end{aligned}
\end{equation}
where $\tau\in\Omega$ is the sampling point. Clearly, in the imaging process this indicator \eqref{eq:passive_direct_scattered_rtm_ind} shall inevitably fail since positions $z_{l}$ are unknown. The above discussed two issues of traditional direct sampling approaches can be overcome by employing \eqref{eq:C} instead of $u^{s}(x_{j},z_{l})$. Thus in what follows, we shall introduce the specific design of our imaging functional. Because of the approximation relationship \eqref{eq:app_c_scattered}, we can naturally benefit from the existing active imaging methods. Comparable to the reverse time migration method, our imaging algorithm contains two phases. The first phase is the back-propagation stage in which the complex conjugated data $\overline{C_{jm}}$ is back-propagated into the domain. Specifically, for $m=1,2,\cdots,J$, we compute the solution $v_{b}$ to the following problem:
\begin{equation}\label{eq:back_helmholtz}
    \begin{aligned}
    \bigtriangleup v_{b}(x,x_{m})+k^{2}v_{b}(x,x_{m})=\frac{2\pi r_{B}}{J}\sum\limits_{j=1}^{J}\overline{C_{jm}}\delta_{x_{j}}(x)\ \ \mbox{in}\  \mathbb{R}^2,
    \end{aligned}
\end{equation}
\begin{equation}\label{eq:back_rc}
    \begin{aligned}
    \lim_{r=|x|\rightarrow\infty}\sqrt{r}(\frac{\partial{v_{b}}}{\partial{r}}-\mathrm{i}kv_{b})=0.
    \end{aligned}
\end{equation}
By \eqref{eq:back_fundamental_helmholtz}, \eqref{eq:back_fundamental_rc} and the linearity, the solution $v_b$ to \eqref{eq:back_helmholtz} and \eqref{eq:back_rc} can be represented by
\begin{equation}\label{eq:back_vb_solution}
    \begin{aligned}
    v_{b}(x,x_{m})=-\frac{2\pi r_{B}}{J}\sum\limits_{j=1}^{J}\overline{C_{jm}}\phi(x,x_{j}).
    \end{aligned}
\end{equation}
Then for $\tau\in\Omega$, the imaginary part of the cross-correlation of the fundamental solution $\phi(\tau,x_{m})$ and the back-propagated field $v_{b}(\tau,x_{m})$ is employed in the second phase, which has the form
\begin{equation}\label{eq:passive_DCM_ind}
    \begin{aligned}
    I(\tau)&=k^{2}\mathrm{Im}\left\{\frac{2\pi r_{B}}{J}\sum\limits_{m=1}^{J}\phi(\tau,x_{m})v^{b}(\tau,x_{m})\right\}\\
    &=-k^{2}\mathrm{Im}\left\{\frac{(2\pi)^{2} r_{B}^{2}}{J^{2}}\sum\limits_{m=1}^{J}\sum\limits_{j=1}^{J}\phi(\tau,x_{m})\phi(\tau,x_{j})\overline{C_{jm}}\right\}.
    \end{aligned}
\end{equation}
The formula \eqref{eq:passive_DCM_ind} is the main imaging functional of the proposed DCM. \eqref{eq:C} and \eqref{eq:passive_DCM_ind} are the so-called two cross-correlations of DCM.

\subsection{The resolution analysis and stability analysis}
In this subsection, we first investigate the resolution analysis of the proposed DCM, which aims to show different properties of indicator functional \eqref{eq:passive_DCM_ind} when probing point $\tau$ approaches the boundary $\partial D$ and goes away from $\partial D$. To this end, in the following we estimate the bound of \eqref{eq:passive_DCM_ind} when $\tau$ changes. Then we further study the stability analysis that indicates our DCM is robust to the added noise. Considering the relationship \eqref{eq:app_c_scattered}, we define
\begin{equation}\label{eq:Ns_discrete_DCM_ind}
    \begin{aligned}
    I_{N^{s}}(\tau)=-k^{2}\mathrm{Im}\left\{\frac{(2\pi)^{2} r_{B}^{2}}{J^{2}}\sum\limits_{m=1}^{J}\sum\limits_{j=1}^{J}\phi(\tau,x_{m})\phi(\tau,x_{j})\overline{N^{s}_{jm}}\right\}.
    \end{aligned}
\end{equation}
If the sampling point $\tau$ is contained inside the curve $\partial{B}$, then $\phi(\tau,x_{j})$ and $\phi(\tau,x_{m})$ are both smooth functions. Furthermore, the imaginary scattered field $N^{s}_{jm}$ is also smooth. As a result, $I_{N^{s}}$ can be regarded as the discrete approximation of the following continuous indicator function:
\begin{equation}\label{eq:Ns_continuous_DCM_ind}
    \begin{aligned}
    \hat{I}_{N^{s}}(\tau)&=-k^{2}\mathrm{Im}\int_{\partial{B}}\int_{\partial{B}}\phi(\tau,x_{m})\phi(\tau,x_{j})\overline{N^{s}_{jm}}\mathrm{d}s(x_{m})\mathrm{d}s(x_{j})\\
    &=-k^{2}\mathrm{Im}\int_{\partial{B}}\int_{\partial{B}}\phi(\tau,x_{m})\phi(\tau,x_{j})\overline{u^{s}(x_{j},x_{m})}\mathrm{d}s(x_{m})\mathrm{d}s(x_{j})\\
    &+k^{2}\mathrm{Im}\int_{\partial{B}}\int_{\partial{B}}\phi(\tau,x_{m})\phi(\tau,x_{j})u^{s}(x_{j},x_{m})\mathrm{d}s(x_{m})\mathrm{d}s(x_{j}).
    \end{aligned}
\end{equation}

Clearly, the first term at the right hand side of \eqref{eq:Ns_continuous_DCM_ind} is the standard indicator of reverse time migration method that is presented in \cite{Chen_zhi_ming2}. Its analysis is also thoroughly studied in \cite{Chen_zhi_ming2}. Thus, for the following resolution analysis, we mainly focus on estimating the second term. To this end, we first present some lemmas for the following use.

\begin{lem}[\cite{DRIA,McLean1}]\label{lem:direct_well_posedness}
Let $\widetilde{\omega}\in H^{1/2}(\partial D)$, the forward obstacle scattering problem:
\begin{equation}\label{eq:helmholtz_omega}
    \begin{aligned}
    \bigtriangleup \omega+k^{2}\omega=0\ \ \mbox{in}\  \mathbb{R}^2 \backslash \overline{D},\ \ \omega=\widetilde{\omega} \ \ \mbox{on}\  \partial{D},
    \end{aligned}
\end{equation}
\begin{equation}\label{eq:rc_omega}
    \begin{aligned}
    \lim_{r=|x|\rightarrow\infty}\sqrt{r}( \frac{\partial{\omega}}{\partial{r}}-\mathrm{i}k\omega)=0,
    \end{aligned}
\end{equation}
admits a unique solution $\omega\in H^{1}_{loc}(\mathbb{R}^2 \backslash \overline{D})$. Furthermore, there exists a constant $C\in\mathbb{R}_{+}$ such that
\begin{equation}
    \begin{aligned}
    \left\|\frac{\partial \omega}{\partial \nu}\right\|_{H^{-1/2}(\partial D)}\leq C\|\widetilde{\omega}\|_{H^{1/2}(\partial D)},
    \nonumber
    \end{aligned}
\end{equation}
where $\nu$ denotes the unit outer normal to $\partial D$.
\end{lem}

Lemma \ref{lem:direct_well_posedness} depicts the well-posedness of the forward scattering problem. We refer to \cite{DRIA,McLean1} for more relevant discussions. Define the Dirichlet-to-Neumann mapping $T:H^{1/2}(\partial D)\rightarrow H^{-1/2}(\partial D)$ for any $\widetilde{\omega}\in H^{1/2}(\partial D)$ by $T(\widetilde{\omega})=\frac{\partial \omega}{\partial \nu}$. It follows from Lemma \ref{lem:direct_well_posedness} that $T$ is a bounded linear operator. Denote the norm by $\|T\|$. Since the obstacle $D$ and the sampling point $\tau$ are both contained inside $\partial B$, we define $h_{1}=\mathrm{dist}(\partial{B}, \partial{D})$ and $h_{2}=\mathrm{dist}(\tau, \partial{B})$. The following lemma on the estimation of the second term at the right hand side of \eqref{eq:Ns_continuous_DCM_ind} holds.

\begin{lem}\label{lem:second_part_bound}
There exists constants $M_{1},M_{2}>0$ such that
\begin{equation}\label{eq:second_part_bound}
    \begin{aligned}
    &\left|k^{2}\mathrm{Im}\int_{\partial{B}}\int_{\partial{B}}\phi(\tau,x_{m})\phi(\tau,x_{j})u^{s}(x_{j},x_{m})\mathrm{d}s(x_{m})\mathrm{d}s(x_{j})\right|\\
    &\leq M_{1}(1+\|T\|)r^{2}_{B}(h_{1}h_{2})^{-1}+M_{2}r^{2}_{B}k^{-1/2}h_{1}^{-3/2}h_{2}^{-1},
    \end{aligned}
\end{equation}
where $M_{1}$ and $M_{2}$ are independent of $k$, $h_{1}$, $h_{2}$, $\|T\|$ and $r_{B}$ but depends on the size of $D$.
\end{lem}
\begin{proof}
From \cite{Chen_zhi_ming3, chandler_wilde1}, we obtain
\begin{equation}\label{eq:hankel_estimate}
    \begin{aligned}
    |H_{0}^{(1)}(t)|\leq\sqrt{\frac{2}{\pi t}},\quad\quad |H_{1}^{(1)}(t)|\leq\sqrt{\frac{2}{\pi t}}+\frac{2}{\pi t}, \quad \forall t> 0,
    \end{aligned}
\end{equation}
where $H_{0}^{(1)}$ is the Hankel function of the first kind of order zero and $H_{1}^{(1)}$ is the Hankel function of the first kind of the first order. By the integral representation formula, the scattered field $u^{s}(x_{j},x_{m})$ is given as
\begin{equation}\label{eq:us_integral}
    \begin{aligned}
    u^{s}(x_{j},x_{m})=\int_{\partial D}\left(u^{s}(y,x_{m})\frac{\partial \phi(x_{j},y)}{\partial \nu(y)}-\frac{\partial u^{s}(y,x_{m})}{\partial \nu(y)}\phi(x_{j},y)\right)\mathrm{d}s(y).
    \end{aligned}
\end{equation}
Combining \eqref{eq:hankel_estimate}, \eqref{eq:us_integral} and the fact that $u^{s}(y,x_{m})|_{y\in \partial D}=-\phi(y,x_{m})|_{y\in \partial D}$, we have
\begin{equation}\label{eq:us_integral_bound}
    \begin{aligned}
    |u^{s}(x_{j},x_{m})|&\leq \frac{|\partial D|}{8\pi}(1+\|T\|)(k|x_{m}-y|)^{-1/2}(k|x_{j}-y|)^{-1/2}\\
    &\quad\quad +\frac{\sqrt{2\pi}|\partial D|}{8\pi^{2}}(k|x_{m}-y|)^{-1/2}(k|x_{j}-y|)^{-1}\\
    &\leq \frac{|\partial D|}{8\pi}(1+\|T\|)(kh_{1})^{-1}+\frac{\sqrt{2\pi}|\partial D|}{8\pi^{2}}(kh_{1})^{-3/2}.
    \end{aligned}
\end{equation}
By \eqref{eq:us_integral_bound} and again using \eqref{eq:hankel_estimate}, we have
\begin{equation}
    \begin{aligned}
    &\left|k^{2}\mathrm{Im}\int_{\partial{B}}\int_{\partial{B}}\phi(\tau,x_{m})\phi(\tau,x_{j})u^{s}(x_{j},x_{m})\mathrm{d}s(x_{m})\mathrm{d}s(x_{j})\right|\\
    \leq &\frac{|\partial D|}{16}(1+\|T\|)r^{2}_{B}h_{1}^{-1}|\tau-x_{m}|^{-1/2}|\tau-x_{j}|^{-1/2}\\
    &\quad\quad +\frac{\sqrt{2\pi}|\partial D|}{16\pi}r^{2}_{B}k^{-1/2}h_{1}^{-3/2}|\tau-x_{m}|^{-1/2}|\tau-x_{j}|^{-1/2}.
    \nonumber
    \end{aligned}
\end{equation}
Then, we can derive the estimate \eqref{eq:second_part_bound}. This completes the proof.
\end{proof}

Now, we are ready to present our main results.

\begin{thm}\label{thm:ind_resolution_rem}
For the sampling point $\tau\in\Omega$ that is located inside $\partial B$, let $\psi(x,\tau)$ be the radiation solution of the Helmholtz equation
\begin{equation}\label{eq:helmholtz_psi}
    \begin{aligned}
    \bigtriangleup \psi(x,\tau)+k^{2}\psi(x,\tau)=0\ \ \mbox{in}\  \mathbb{R}^2 \backslash \overline{D},\ \ \psi(x,\tau)=-\mathrm{Im}\phi(x,\tau) \ \ \mbox{on}\  \partial{D}.
    \end{aligned}
\end{equation}
Assume that the position $z_{l}$ of the random incident source is $\xi$-perturbed, where $\xi$ denotes the perturbation value. Then for $0<\xi<1/2$, $I(\tau)$ defined in \eqref{eq:passive_DCM_ind} satisfies that
\begin{equation}\label{eq:active_rtm_continue_psi}
    \begin{aligned}
    I(\tau)=k\int_{S}|\psi^{\infty}(\hat{x},\tau)|^{2}\mathrm{d}\hat{x}+R(\tau),
    \end{aligned}
\end{equation}
where $S:=\left\{\hat{x} \in \mathbb{R}^2:|\hat{x}|=1\right\}$ and $R(\tau)$ is the remainder term that satisfies
\begin{equation}\label{eq:R_bound}
    \begin{aligned}
    |R(\tau)|\leq M_{1}(1+\|T\|)r^{2}_{B}(h_{1}h_{2})^{-1}+M_{2}r^{2}_{B}k^{-1/2}h_{1}^{-3/2}h_{2}^{-1} + \frac{M_{3}}{L^{2-4\xi}}+\frac{M_{4}}{J} + M_{5}r_{B}^{-1},
    \end{aligned}
\end{equation}
where $L$ is the number of incident sources, $J$ is the number of measurement points, $M_{1},M_{2},M_{3},M_{4},M_{5}$ are all positive constants and $M_{1},M_{2}$ are defined in \eqref{eq:second_part_bound}.
\end{thm}
\begin{proof}
For $0<\xi<1/2$, by Theorem 1 in \cite{Austin_Trefethen1} and the fact that the scattered field $u^{s}$ is twice continuously differentiable, we have
\begin{equation}\label{eq:Ns_C_error_bound}
    \begin{aligned}
    |N^{s}_{jm}-C_{jm}|=O(\frac{1}{L^{2-4\xi}}).
    \end{aligned}
\end{equation}
Moreover, since \eqref{eq:Ns_discrete_DCM_ind} is the trapezoid quadrature approximation of \eqref{eq:Ns_continuous_DCM_ind}, and $\phi(\tau,x_{j})$, $\phi(\tau,x_{m})$, $N^{s}_{jm}$ are all smooth functions, it follows that
\begin{equation}\label{eq:Ind_Ns_error_bound}
    \begin{aligned}
    &|I_{N^{s}}(\tau)-\hat{I}_{N^{s}}(\tau)|\\
    &=|\frac{1}{2}\frac{(2\pi)^{3}}{J}\frac{\mathrm{d}(\phi(\tau,x)\phi(\tau,y)\overline{u^{s}(x,y)-\overline{u^{s}(x,y)}})}{\mathrm{d}x}|_{(x,y)=(\zeta_{1},\zeta_{2})}\\
    &\quad\quad\quad+\frac{1}{2}\frac{(2\pi)^{3}}{J}\frac{\mathrm{d}(\phi(\tau,x)\phi(\tau,y)\overline{u^{s}(x,y)-\overline{u^{s}(x,y)}})}{\mathrm{d}y}|_{(x,y)=(\zeta_{3},\zeta_{4})}|\\
    &=O(\frac{1}{J}).
    \end{aligned}
\end{equation}

Then, combining Theorem 3.3 of \cite{Chen_zhi_ming2}, \eqref{eq:Ns_C_error_bound}, \eqref{eq:Ind_Ns_error_bound}, Lemma \ref{lem:second_part_bound}, we obtain
\begin{equation}
    \begin{aligned}
    I(\tau)=k\int_{S}|\psi^{\infty}(\hat{x},\tau)|^{2}\mathrm{d}\hat{x}+R(\tau),
    \nonumber
    \end{aligned}
\end{equation}
and $R(\tau)$ satisfies \eqref{eq:R_bound}. The proof is completed.
\end{proof}

\begin{rem}\label{rem_resolution}
Theorem \ref{thm:ind_resolution_rem} indicates that the indicator functional \eqref{eq:passive_DCM_ind} has the contrast at the boundary $\partial D$ and decays when the sampling point $\tau$ goes away from the obstacle $D$. Additionally, we would like to point out that, Theorem \ref{thm:ind_resolution_rem} holds for $0<\xi<1/2$ primarily due to \eqref{eq:Ns_C_error_bound}. Nonetheless, numerical experiments presented in Section 4 show that our proposed approach also works well for the case of $\xi\geq 1/2$.
\end{rem}

Let $\widetilde{h}_{1}=\mathrm{dist}(\partial{D},\Sigma)$ and $\widetilde{h}_{2}=\mathrm{dist}(\partial{B},\Sigma)$ respectively be distances of $\partial D$ to $\Sigma$ and $\partial B$ to $\Sigma$. The following stability result holds.

\begin{thm}\label{thm:stability_rem}
There exist positive constants $M_{6}$ and $M_{7}$ that are independent of the sampling point $\tau$ such that
\begin{equation}\label{eq:stability_result}
    \begin{aligned}
    \left|I(\tau)-I^{\delta}(\tau)\right|
    \leq M_{6}\|u^{\delta}(x,z)-u(x,z)\|^{2}_{L^{\infty}(\partial B \times \Sigma)} + M_{7}\|u^{\delta}(x,z)-u(x,z)\|_{L^{\infty}(\partial B \times \Sigma)},
    \end{aligned}
\end{equation}
where $I^{\delta}(\tau)$ is the indicator functional \eqref{eq:passive_DCM_ind} with $u(x_{j},z_{l})$ and $u(x_{m},z_{l})$ replaced by $u^{\delta}(x_{j},z_{l})$ and $u^{\delta}(x_{m},z_{l})$.
\end{thm}
\begin{proof}
Using \eqref{eq:hankel_estimate} and the integral representation formula, we have
\begin{equation}
    \begin{aligned}
    |u^{s}(x_{m},z_{l})|\leq \frac{|\partial D|}{8\pi}(1+\|T\|)(k\widetilde{h}_{1})^{-1/2}(kh_{1})^{-1/2}+\frac{\sqrt{2\pi}|\partial D|}{8\pi^{2}}(k\widetilde{h}_{1})^{-1/2}(kh_{1})^{-1}.
    \nonumber
    \end{aligned}
\end{equation}
and
\begin{equation}
    \begin{aligned}
    |u^{s}(x_{j},z_{l})|\leq \frac{|\partial D|}{8\pi}(1+\|T\|)(k\widetilde{h}_{1})^{-1/2}(kh_{1})^{-1/2}+\frac{\sqrt{2\pi}|\partial D|}{8\pi^{2}}(k\widetilde{h}_{1})^{-1/2}(kh_{1})^{-1}.
    \nonumber
    \end{aligned}
\end{equation}
Again using \eqref{eq:hankel_estimate}, one can further obtain
\begin{equation}
    \begin{aligned}
    &\left|u(x_{m},z_{l})\right|=\left|u^{s}(x_{m},z_{l})+u^{i}(x_{m},z_{l})\right|\\
    \leq &\frac{|\partial D|}{8\pi}(1+\|T\|)(k\widetilde{h}_{1})^{-1/2}(kh_{1})^{-1/2}+\frac{\sqrt{2\pi}|\partial D|}{8\pi^{2}}(k\widetilde{h}_{1})^{-1/2}(kh_{1})^{-1} + \frac{\sqrt{2\pi}}{4\pi}(k\widetilde{h}_{2})^{-1/2}.
    \nonumber
    \end{aligned}
\end{equation}
Analogously,
\begin{equation}
    \begin{aligned}
    \left|u(x_{j},z_{l})\right|\leq \frac{|\partial D|}{8\pi}(1+\|T\|)(k\widetilde{h}_{1})^{-1/2}(kh_{1})^{-1/2}+\frac{\sqrt{2\pi}|\partial D|}{8\pi^{2}}(k\widetilde{h}_{1})^{-1/2}(kh_{1})^{-1} + \frac{\sqrt{2\pi}}{4\pi}(k\widetilde{h}_{2})^{-1/2}.
    \nonumber
    \end{aligned}
\end{equation}
Define $M=\frac{|\partial D|}{8\pi}(1+\|T\|)(k\widetilde{h}_{1})^{-1/2}(kh_{1})^{-1/2}+\frac{\sqrt{2\pi}|\partial D|}{8\pi^{2}}(k\widetilde{h}_{1})^{-1/2}(kh_{1})^{-1} + \frac{\sqrt{2\pi}}{4\pi}(k\widetilde{h}_{2})^{-1/2}$. Therefore, it holds that
\begin{equation}\label{eq:u_conj_minus_u}
    \begin{aligned}
    &\left|\overline{u^{\delta}(x_{j},z_{l})}u^{\delta}(x_{m},z_{l})-\overline{u(x_{j},z_{l})}u(x_{m},z_{l})\right|\\
    &=|(u^{\delta}(x_{m},z_{l})-u(x_{m},z_{l}))\overline{u^{\delta}(x_{j},z_{l})-u(x_{j},z_{l})}\\
    &\quad\quad\quad +u(x_{m},z_{l})\overline{u^{\delta}(x_{j},z_{l})-u(x_{j},z_{l})} + \overline{u(x_{j},z_{l})}(u^{\delta}(x_{m},z_{l})-u(x_{m},z_{l}))|\\
    &\leq \|u^{\delta}(x,z)-u(x,z)\|^{2}_{L^{\infty}(\partial B \times \Sigma)} + 2M\|u^{\delta}(x,z)-u(x,z)\|_{L^{\infty}(\partial B \times \Sigma)}.
    \end{aligned}
\end{equation}
Then, by \eqref{eq:u_conj_minus_u}, we have
\begin{equation}
    \begin{aligned}
    \left|I(\tau)-I^{\delta}(\tau)\right|
    \leq \frac{\pi k^{2}r_{B}^{2}|\Sigma|}{h_{2}}(\|u^{\delta}(x,z)-u(x,z)\|^{2}_{L^{\infty}(\partial B \times \Sigma)} + 2M\|u^{\delta}(x,z)-u(x,z)\|_{L^{\infty}(\partial B \times \Sigma)}),
    \nonumber
    \end{aligned}
\end{equation}
which gives rise to \eqref{eq:stability_result}. This completes the proof.
\end{proof}

\section{Numerical experiments}
In this section, we shall conduct some numerical examples for the passive inverse scattering to demonstrate the effectiveness of our proposed DCM.  In all the following examples, the simulated noisy  total field $u^{\delta}(x,z_{l})$ is generated by the following formula:
\begin{equation}
\begin{aligned}
u^{\delta}(x,z_{l})=u(x,z_{l})(1+\delta \Delta),
\nonumber
\end{aligned}
\end{equation}
where $\delta$ is the noise level and $\Delta$ is a random number drawn from the normal distribution $\mathcal{N}(0,1)$. Here, the total field $u(x,z_{l})$ is solved by the boundary integral equation method in the form of the double-layer potential. Accordingly, $C_{jm}^{\delta}$ is the designed cross-correlation \eqref{eq:C} with $u(x,z_{l})$ replaced by $u^{\delta}(x,z_{l})$. The curve $\Sigma$ is selected as a circle of radius 100 centered at origin. Incident sources are randomly distributed on $\Sigma$ and their positions are
\begin{align*}
    z_{l} = 100(\cos \theta_{l}, \sin \theta_{l})^{T}, \quad\quad \theta_{l} = \frac{2\pi}{L}(l-1+\xi_{l}), 1\leq l\leq L,
\end{align*}
where $\xi_{l}$ is a random value drawn from $U[0,\xi]$. The radius of the measurement curve $\partial B$ that encloses the obstacle $D$ is chosen as $r_{B}=5$ and measurement point $x_{j}$ has the form
\begin{align*}
    x_{j} = 5(\cos \gamma_{j}, \sin \gamma_{j})^{T}, \quad\quad \gamma_{j} = \frac{2\pi}{J}(j-1), 1\leq j\leq J.
\end{align*}
For the sake of convenience , we set $L=J$. Furthermore, the search domain $\Omega$ is $[-5,5]\times[-5,5]$ with $200 \times 200$ equally spaced sampling points.

To demonstrate the superiority of our proposed DCM, we consider the following two cases. The first case is the single obstacle, which is used to discuss the effects of the perturbation $\xi$, the number $L$, the wavenumber $k$, and the noise level $\delta$ on the reconstruction. This case allows us to compare our results to those in \cite{GHM1}. The second case is multiple obstacles, which is used to test our method's ability to handle the union of two obstacles that are close together and the multiscale case.

\subsection{Example 1: Single obstacle case}

\begin{figure}[t]
	\centering
	\subfigure[$\mathrm{Im}(C_{jm}),\xi=0.4$]{
		\includegraphics[width=0.3\textwidth]{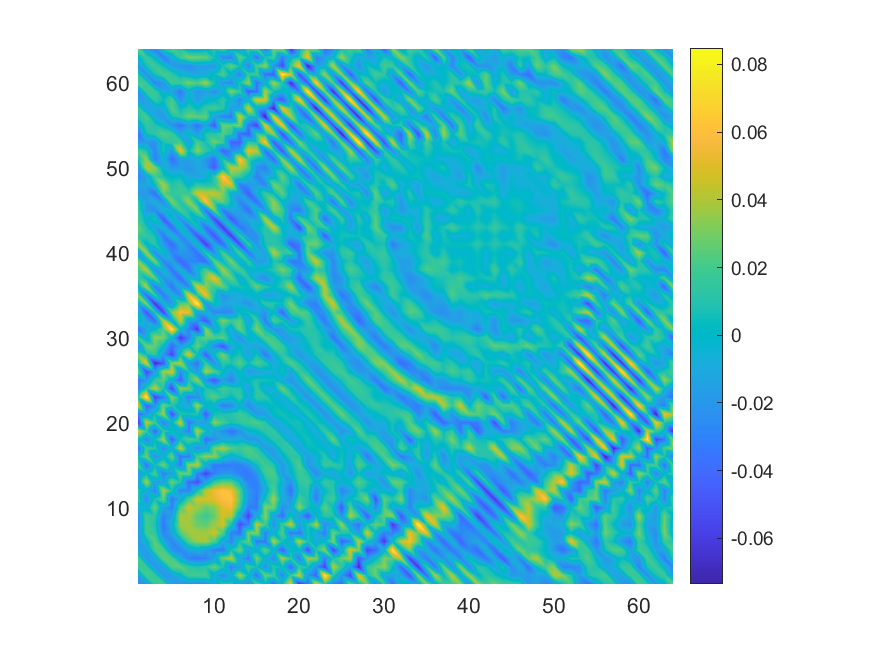}
	}
	\subfigure[$\mathrm{Im}(C_{jm}),\xi=2$]{
		\includegraphics[width=0.3\textwidth]{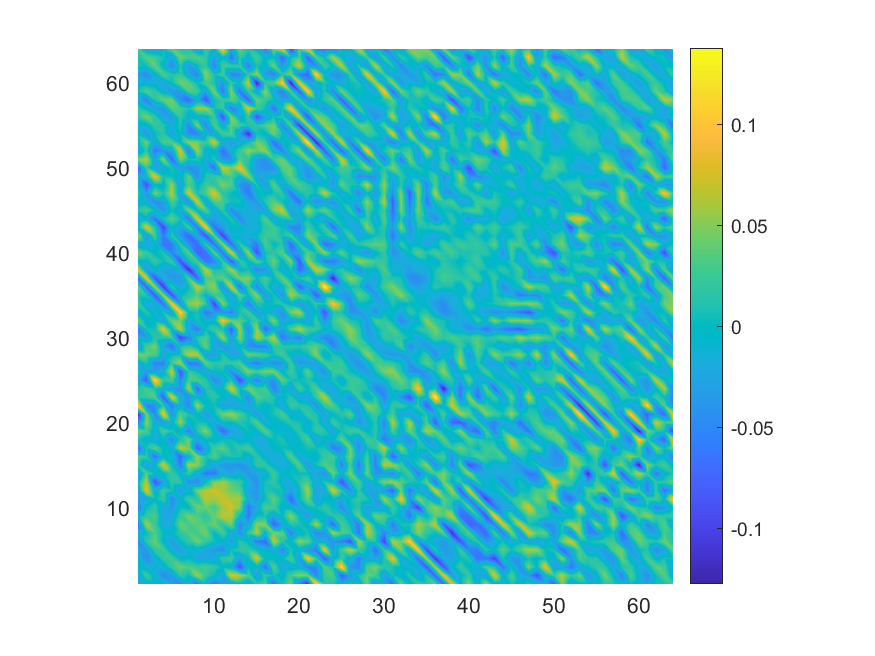}
	}
	\subfigure[$\mathrm{Im}(N_{jm}^{s})$]{
		\includegraphics[width=0.3\textwidth]{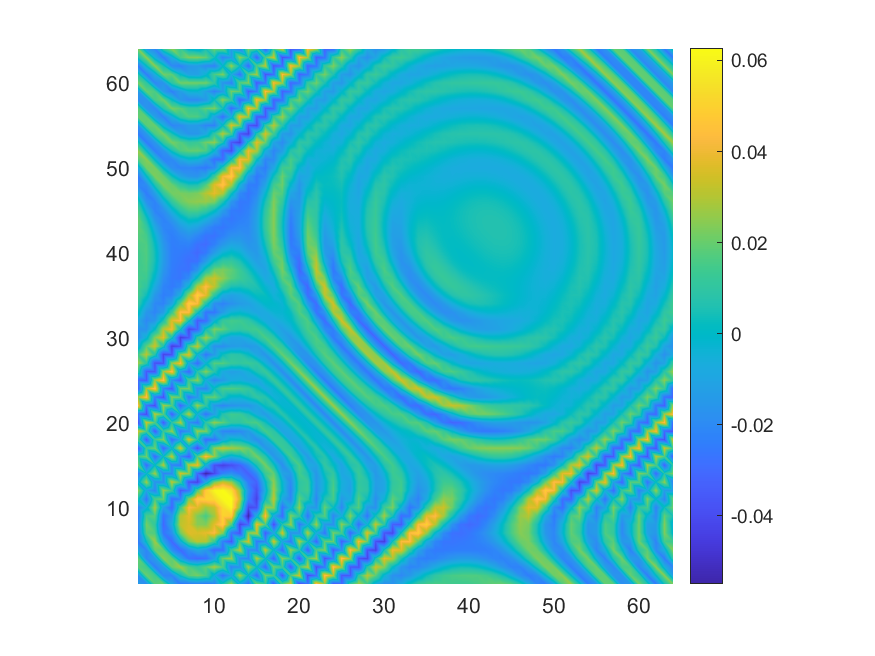}
	}\\
	\subfigure[$\mathrm{Im}(C_{jm}),\xi=0.4$]{
		\includegraphics[width=0.3\textwidth]{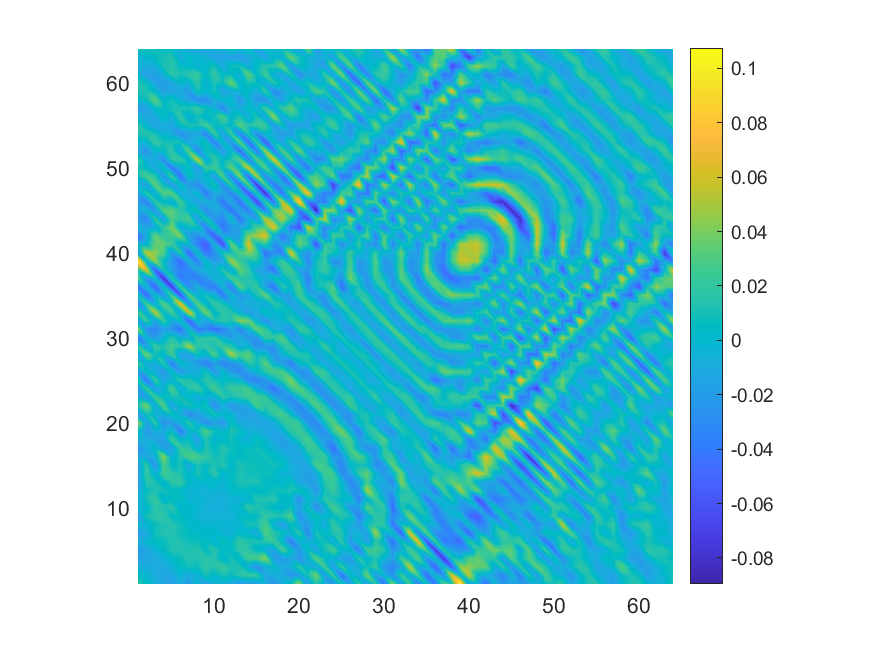}
	}
	\subfigure[$\mathrm{Im}(C_{jm}),\xi=2$]{
		\includegraphics[width=0.3\textwidth]{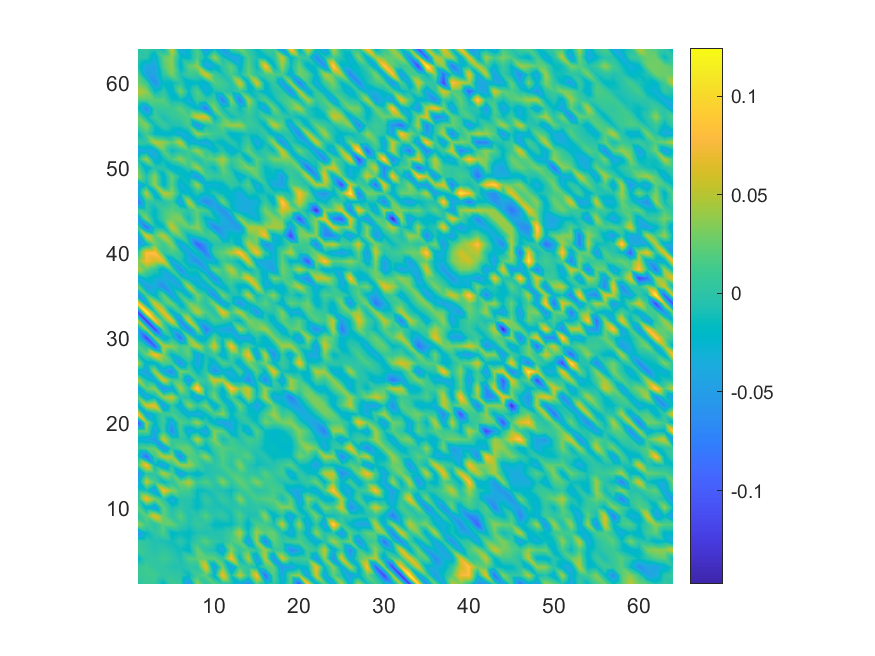}
	}
	\subfigure[$\mathrm{Im}(N_{jm}^{s})$]{
		\includegraphics[width=0.3\textwidth]{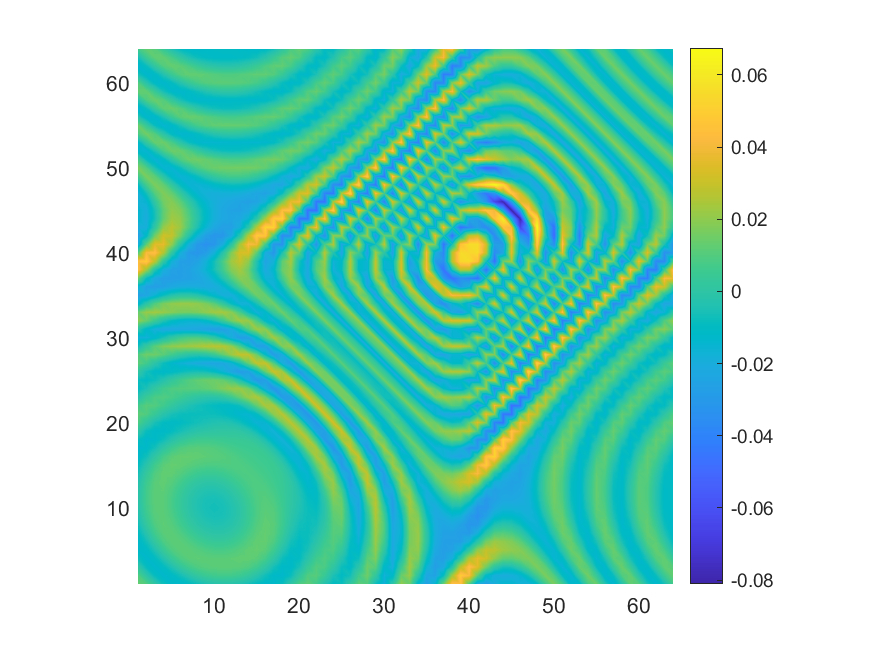}
	}
	\caption{\label{fig:ex1_Cjm_L64}The imaginary parts of $C_{jm}$ and $N^{s}_{jm}$ with $\xi=0.4$ and $\xi=2$ in the case of $L=64$ for the kite (Top) and the peanut (Bottom). The first column displays the imaginary parts of $C_{jm}$ with $\xi=0.4$, the second column displays the imaginary parts of $C_{jm}$ with $\xi=2$ and the third column displays the imaginary parts of $N^{s}_{jm}$.}
\end{figure}

\begin{figure}[h]
	\centering
	\subfigure[Exact shape]{
		\includegraphics[width=0.3\textwidth]{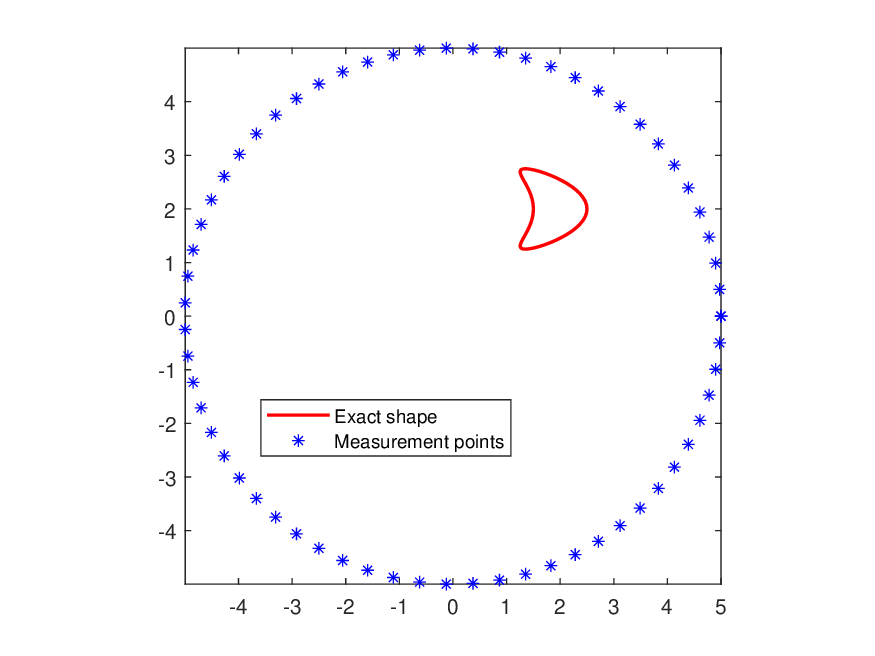}
	}
	\subfigure[$\xi=0.4$]{
		\includegraphics[width=0.3\textwidth]{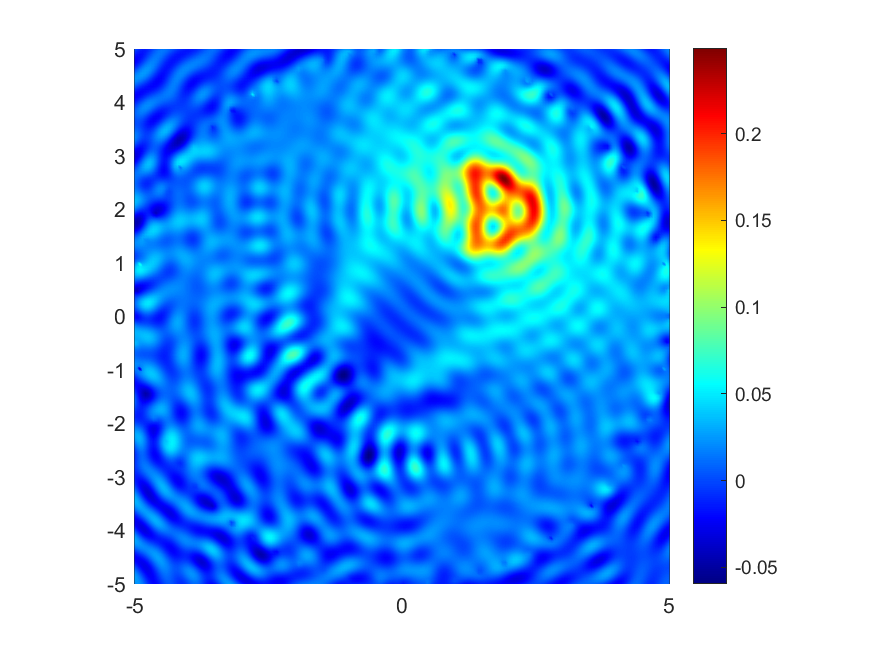}
	}
	\subfigure[$\xi=2$]{
		\includegraphics[width=0.3\textwidth]{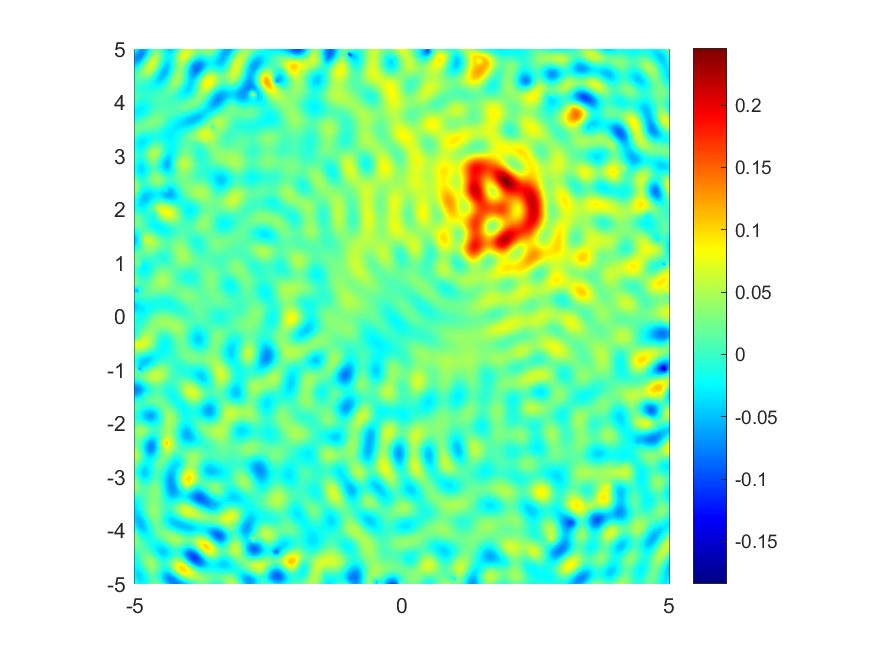}
	}\\
	\subfigure[Exact shape]{
		\includegraphics[width=0.3\textwidth]{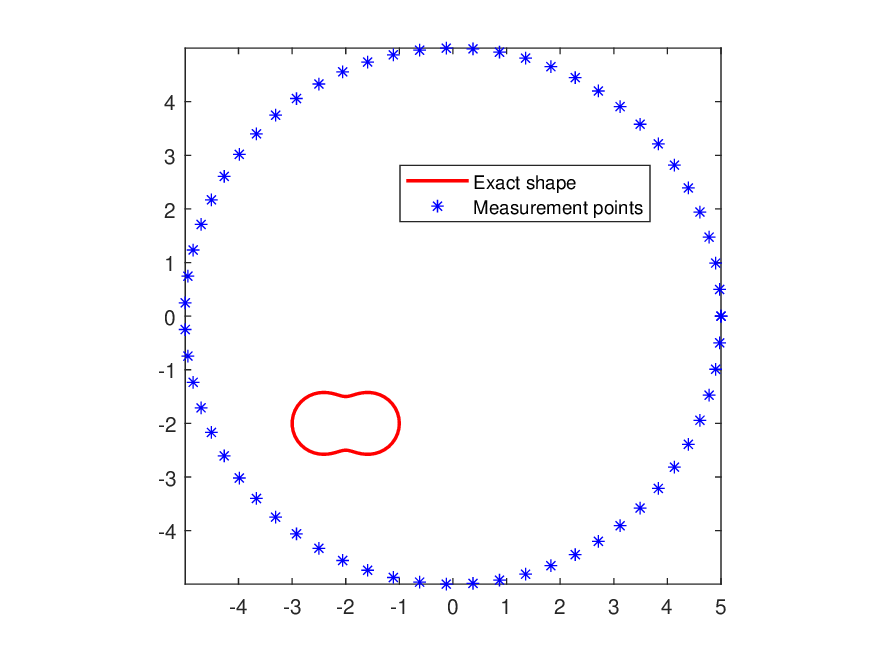}
	}
	\subfigure[$\xi=0.4$]{
		\includegraphics[width=0.3\textwidth]{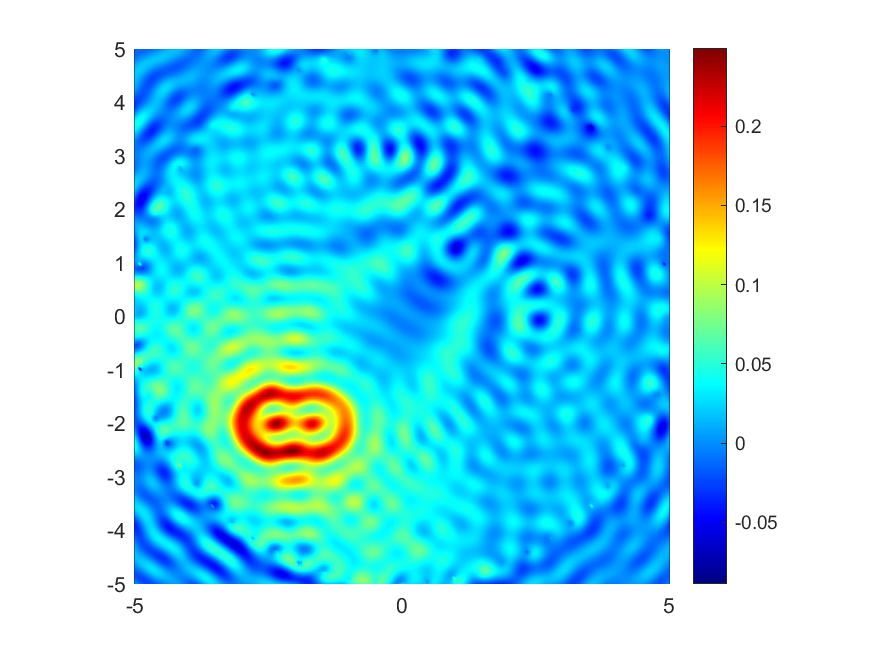}
	}
	\subfigure[$\xi=2$]{
		\includegraphics[width=0.3\textwidth]{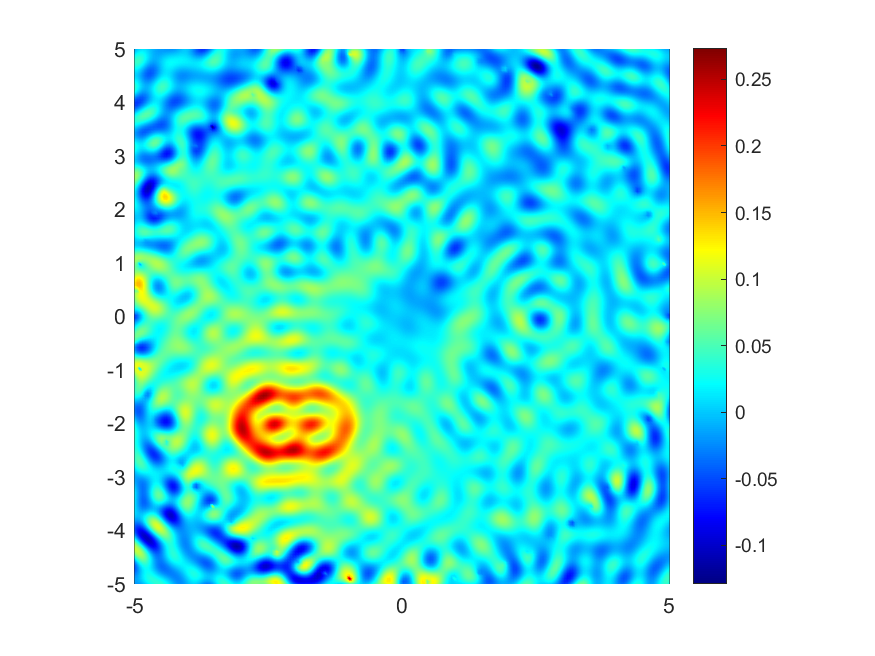}
	}
    \caption{\label{fig:ex1_curve_L64}Recoveries of $\partial D$ by DCM with $\xi=0.4$ and $\xi=2$ in the case of $L=64$ for the kite (Top) and the peanut (Bottom). The first column shows the exact boundary $\partial D$ and $L=64$ measurement points, the second column shows the reconstructed $\partial D$ with $\xi=0.4$ and the third column shows the reconstructed $\partial D$ with $\xi=2$.}
\end{figure}

\begin{figure}[h]
	\centering
	\subfigure[$\mathrm{Im}(C_{jm}),\xi=0.4$]{
		\includegraphics[width=0.3\textwidth]{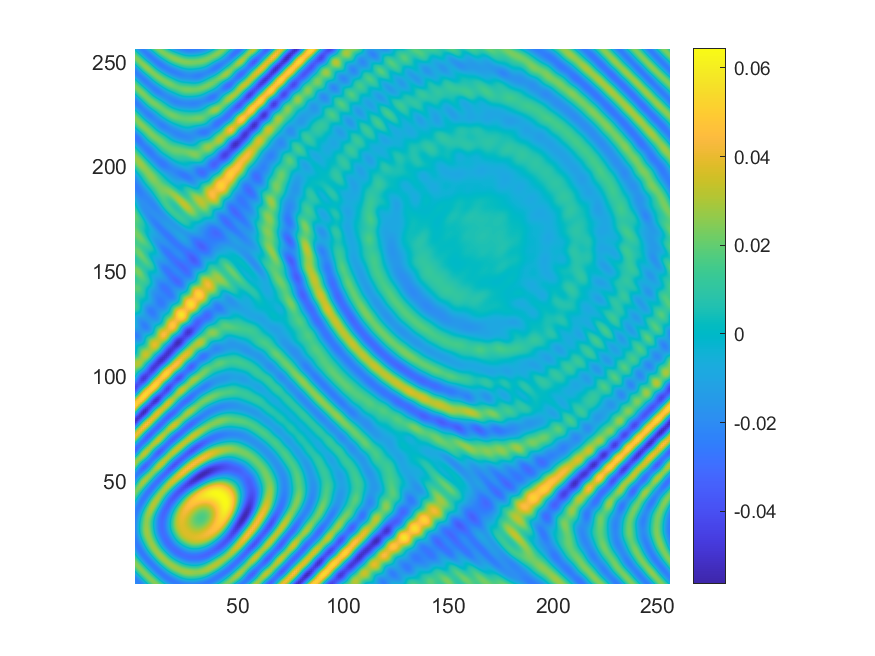}
	}
	\subfigure[$\mathrm{Im}(C_{jm}),\xi=2$]{
		\includegraphics[width=0.3\textwidth]{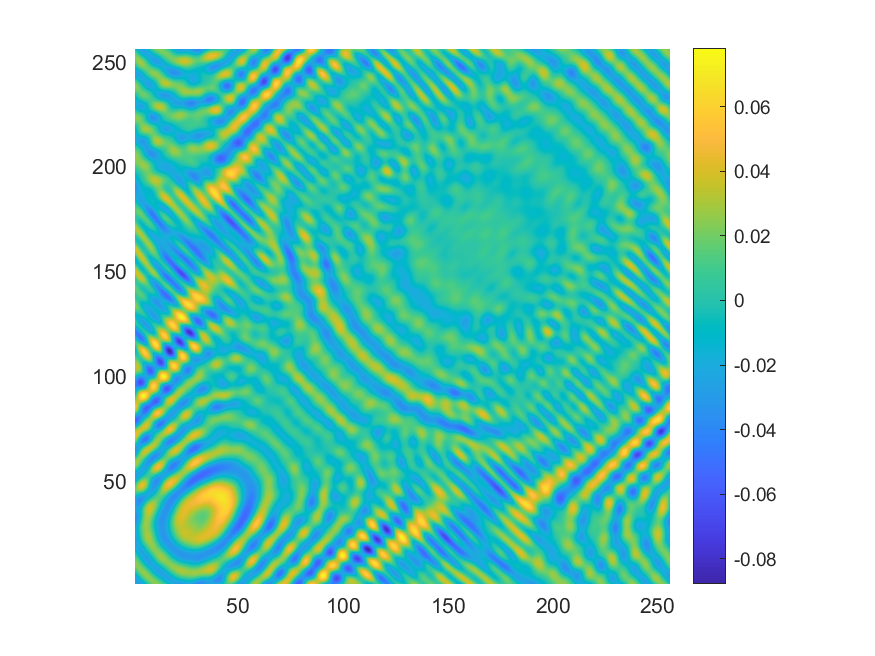}
	}
	\subfigure[$\mathrm{Im}(N_{jm}^{s})$]{
		\includegraphics[width=0.3\textwidth]{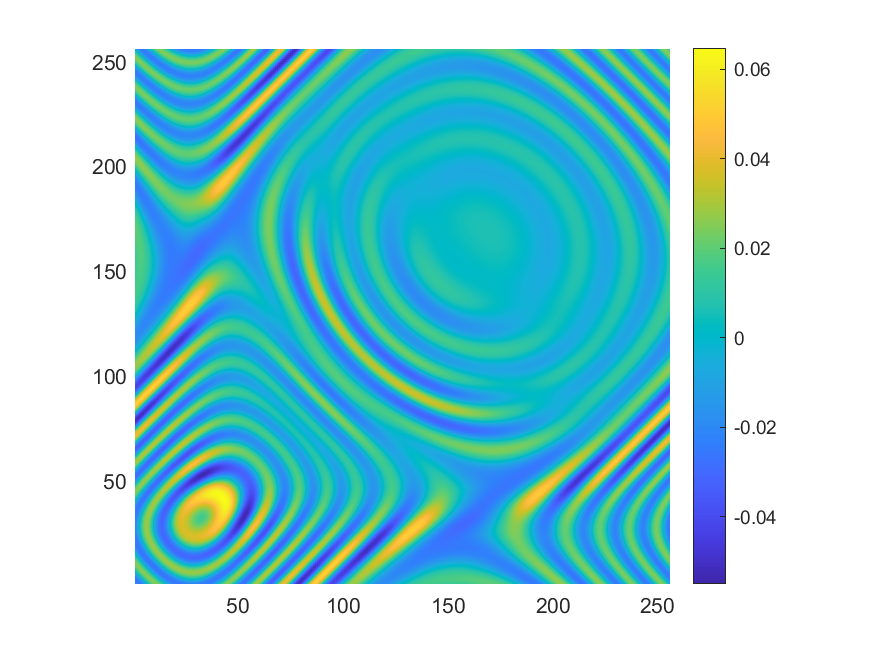}
	}\\
	\subfigure[$\mathrm{Im}(C_{jm}),\xi=0.4$]{
		\includegraphics[width=0.3\textwidth]{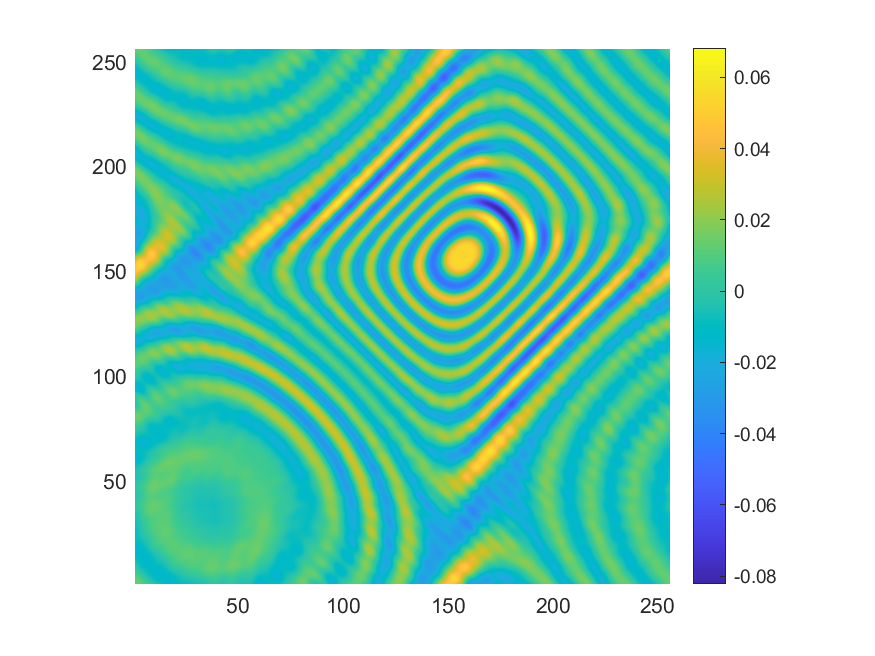}
	}
	\subfigure[$\mathrm{Im}(C_{jm}),\xi=2$]{
		\includegraphics[width=0.3\textwidth]{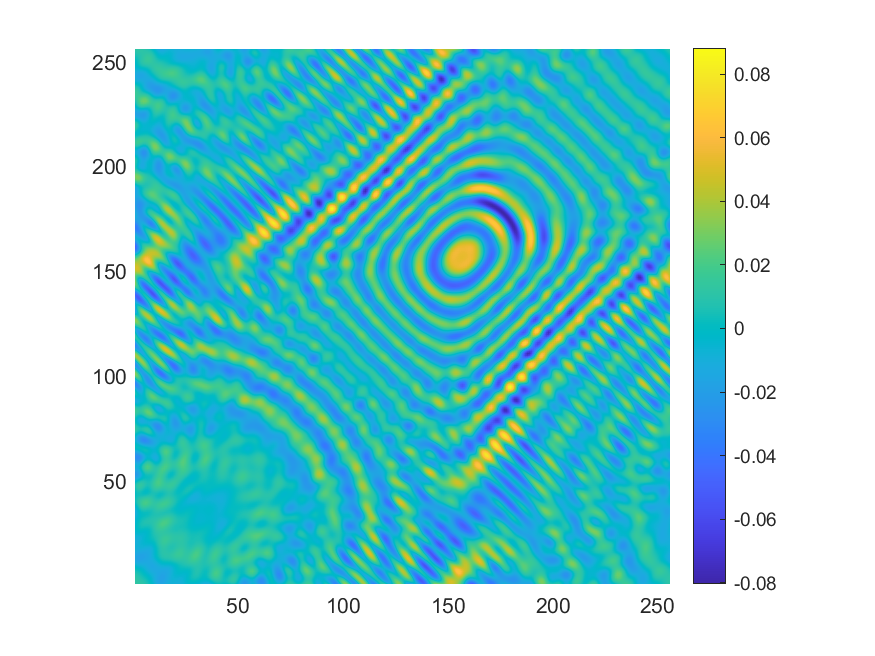}
	}
	\subfigure[$\mathrm{Im}(N_{jm}^{s})$]{
		\includegraphics[width=0.3\textwidth]{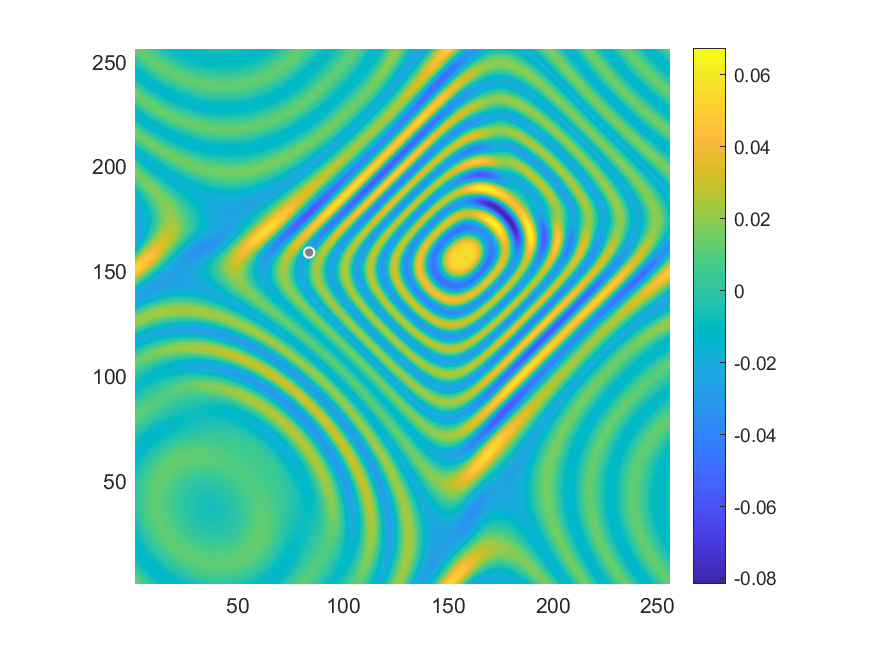}
	}
	\caption{\label{fig:ex1_Cjm_L256}The imaginary parts of $C_{jm}$ and $N^{s}_{jm}$ with $\xi=0.4$ and $\xi=2$ in the case of $L=256$ for the kite (Top) and the peanut (Bottom). The first column displays the imaginary parts of $C_{jm}$ with $\xi=0.4$, the second column displays the imaginary parts of $C_{jm}$ with $\xi=2$ and the third column displays the imaginary parts of $N^{s}_{jm}$.}
\end{figure}

\begin{figure}[h]
	\centering
	\subfigure[Exact shape]{
		\includegraphics[width=0.3\textwidth]{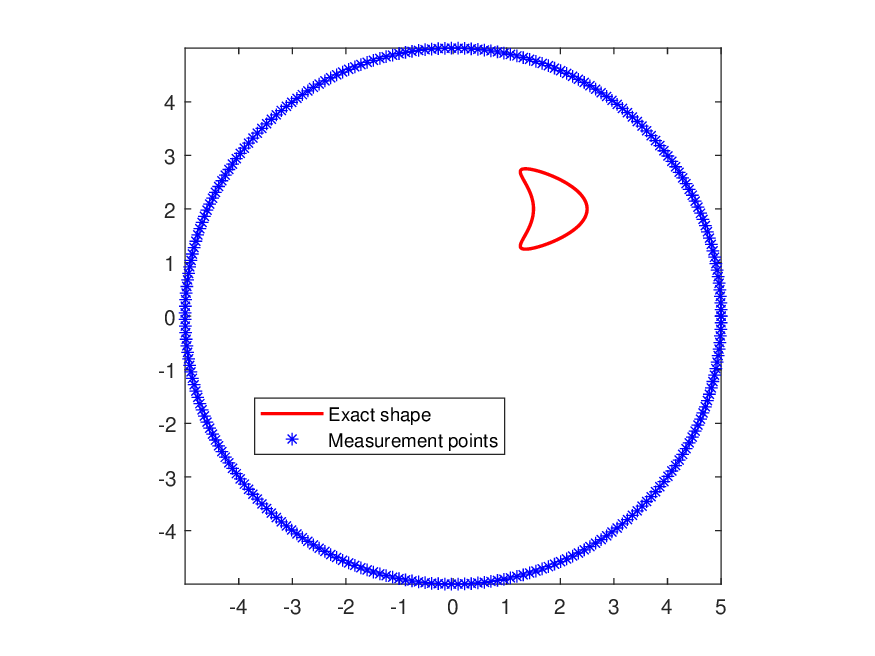}
	}
	\subfigure[$\xi=0.4$]{
		\includegraphics[width=0.3\textwidth]{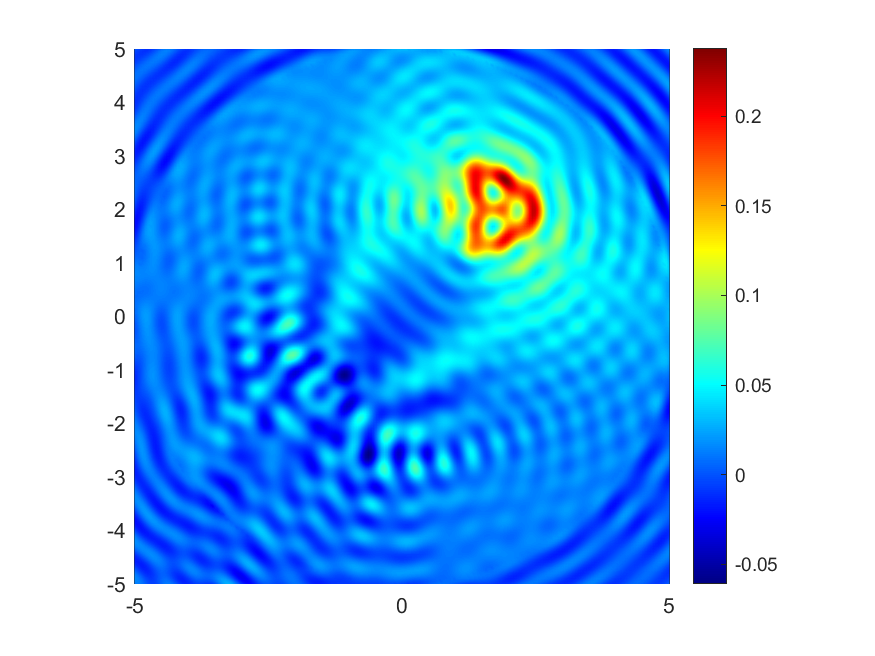}
	}
	\subfigure[$\xi=2$]{
		\includegraphics[width=0.3\textwidth]{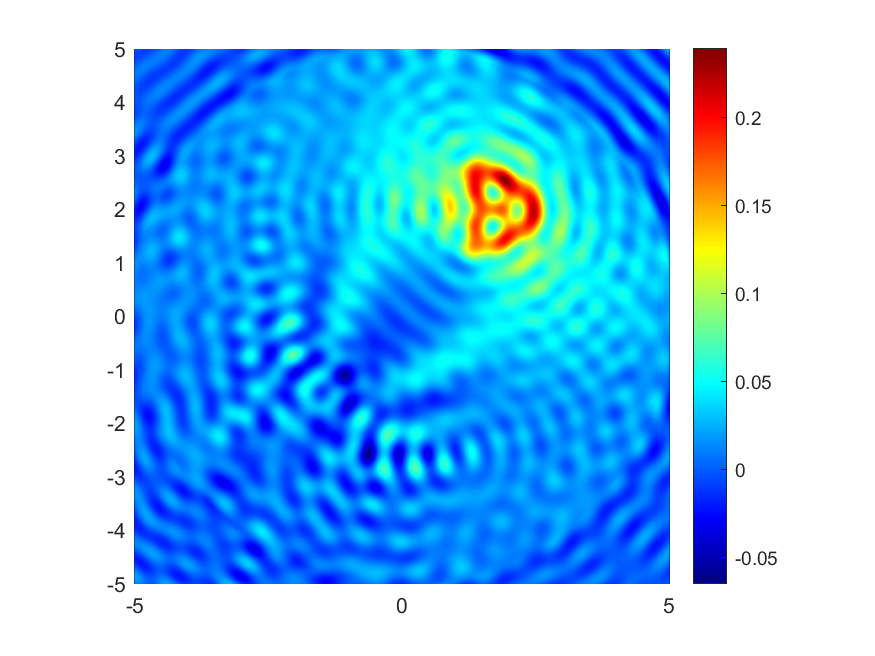}
	}\\
	\subfigure[Exact shape]{
		\includegraphics[width=0.3\textwidth]{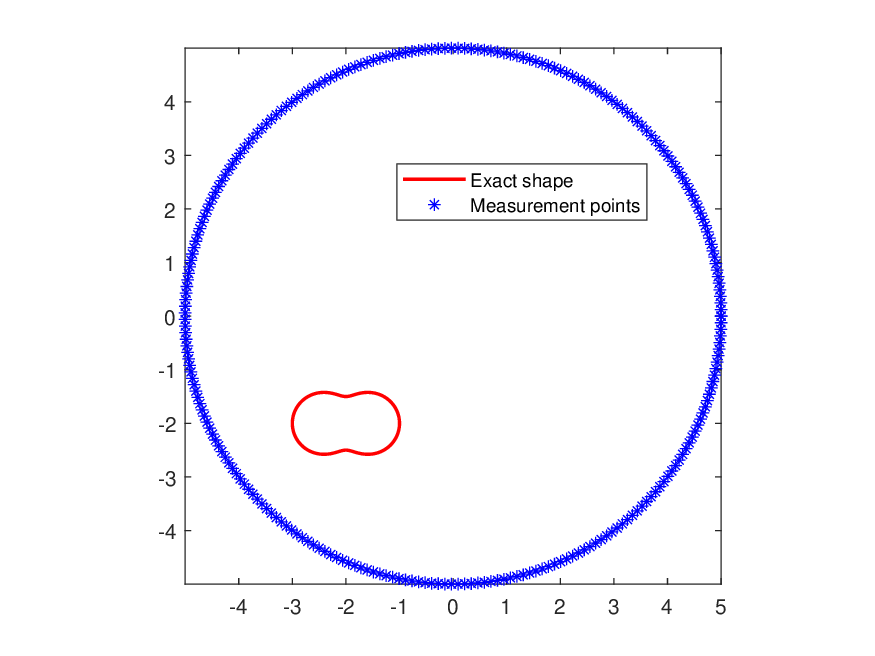}
	}
	\subfigure[$\xi=0.4$]{
		\includegraphics[width=0.3\textwidth]{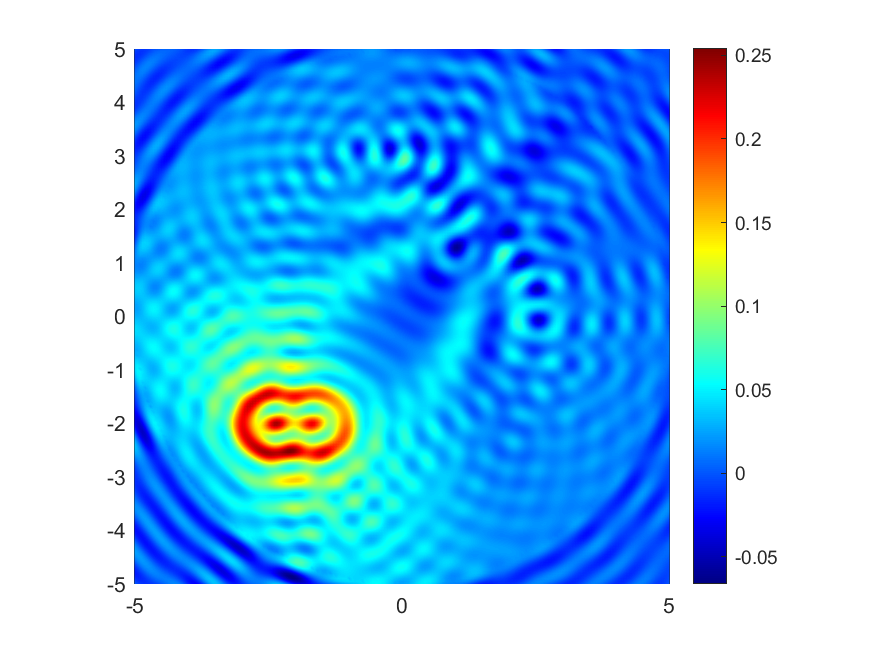}
	}
	\subfigure[$\xi=2$]{
		\includegraphics[width=0.3\textwidth]{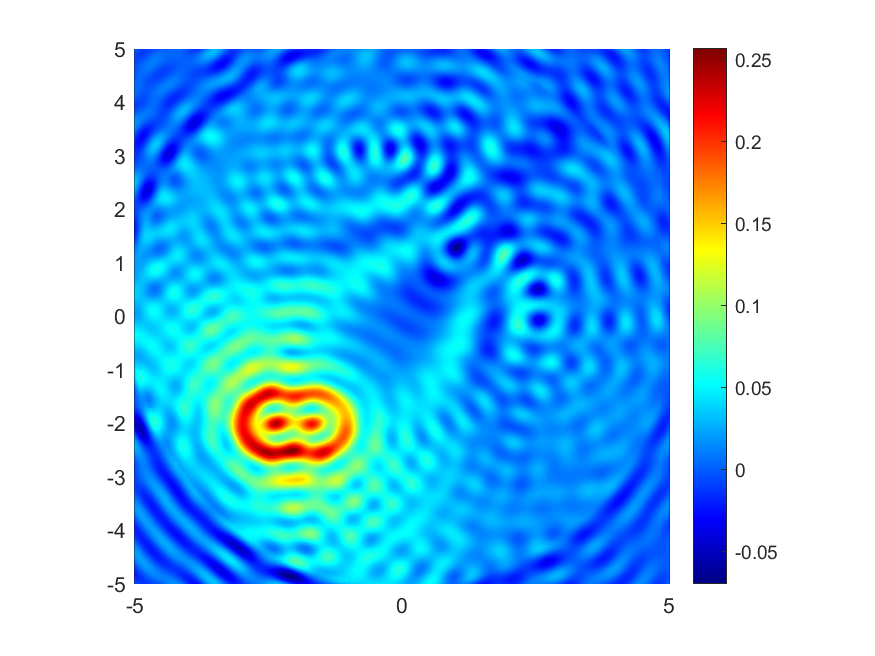}
	}
    \caption{\label{fig:ex1_curve_L256}Recoveries of $\partial D$ by DCM with $\xi=0.4$ and $\xi=2$ in the case of $L=256$ for the kite (Top) and the peanut (Bottom). The first column shows the exact boundary $\partial D$ and $L=256$ measurement points, the second column shows the reconstructed $\partial D$ with $\xi=0.4$ and the third column shows the reconstructed $\partial D$ with $\xi=2$.}
\end{figure}

In this example, we are concerned with the imaging of the single obstacle. To this end, two different obstacles are considered, which are respectively kite-shaped and peanut-shaped. Their boundaries are respectively given by
\begin{equation}
\begin{aligned}
\partial D = \Big\{\big(2+0.5(\cos \theta + 0.65\cos 2\theta - 0.65), 2+0.75\sin \theta\big):\theta\in [0,2\pi]\Big\}.
\nonumber
\end{aligned}
\end{equation}
and
\begin{equation}
\begin{aligned}
\partial D = \Big\{\big(-2+\sqrt{\cos^{2}\theta+0.25\sin^{2}\theta}\cos\theta, -2+\sqrt{\cos^{2}\theta+0.25\sin^{2}\theta}\sin\theta\big):\theta\in [0,2\pi]\Big\}.
\nonumber
\end{aligned}
\end{equation}
This kite-shaped obstacle is also considered in \cite{GHM1}, allowing us to compare our reconstructions with those of \cite{GHM1}.
\subsubsection{The influence of the perturbation $\xi$ and the number $L$ of incident random sources}
We now consider to investigate the influence of the perturbation $\xi$ and the number $L$ of random point sources on the solution of the inverse problem by using noise-free data, i.e., $\delta=0$. The wavenumber is $k=2\pi$.

We first look at how $\xi$ affects the reconstruction. For this purpose, we set $\xi=0.4$ and $\xi=2$, respectively. Moreover, we choose $L=64$. The imaginary parts of the cross-correlations $C_{jm}$ with $\xi=0.4$ and $\xi=2$ for the kite are shown in Fig. \ref{fig:ex1_Cjm_L64}(a) and (b), and the imaginary parts of the cross-correlations $C_{jm}$ for the peanut are shown in Fig. \ref{fig:ex1_Cjm_L64}(d) and (e). The corresponding approximate imaginary parts of $N^{s}_{jm}$ for the kite and peanut are respectively presented in Fig. \ref{fig:ex1_Cjm_L64}(c) and (f). Fig. \ref{fig:ex1_Cjm_L64} shows that larger $\xi$ will generate larger approximation error between $C_{jm}$ and $N^{s}_{jm}$. The boundary reconstructions of our proposed DCM with $\xi=0.4$ and $\xi=2$ are shown in Fig. \ref{fig:ex1_curve_L64}. It is clear in Fig. \ref{fig:ex1_curve_L64} that, our proposed DCM can obtain satisfactory reconstructions. Moreover, Fig. \ref{fig:ex1_curve_L64} indicates that as $\xi$ increases, the reconstruction of the location and shape will be accordingly distorted. This is reasonable because $C_{jm}$ is the discrete approximation of $N^{s}_{jm}$ and the corresponding discrete error that will affect the reconstruction deeply relies on $\xi$. It is worth noticing that our proposed DCM also works well for $\xi=2$ although Theorem \ref{thm:ind_resolution_rem} does not cover $\xi\geq 1/2$.

It is expected by Theorem \ref{thm:ind_resolution_rem} that the inversion quality can be improved by more incident random sources. Hence, the total number $L$ is increased to be $L=256$. In the case of $L=256$, the imaginary parts of the corresponding cross-correlations $C_{jm}$ with $\xi=0.4$ and $\xi=2$ for the kite can be seen in Fig. \ref{fig:ex1_Cjm_L256}(a) and (b), and the imaginary parts of the cross-correlations $C_{jm}$ for the peanut are shown Fig. \ref{fig:ex1_Cjm_L256}(d) and (e). The approximate imaginary parts of $N^{s}_{jm}$ are shown in Fig. \ref{fig:ex1_Cjm_L256}(c) and (f). We find from Fig. \ref{fig:ex1_Cjm_L256} that compared with the results in Fig. \ref{fig:ex1_Cjm_L64}, the discrete error between $C_{jm}$ and $N^{s}_{jm}$ is significantly reduced particularly for $\xi=0.4$. The reconstructions of the kite and peanut for $L=256$ are presented in Fig. \ref{fig:ex1_curve_L256}. As expected, in Fig. \ref{fig:ex1_curve_L256} the reconstructions are improved when $L$ increases. This confirms Theorem \ref{thm:ind_resolution_rem}. Therefore, to meet the required conditions of Theorem \ref{thm:ind_resolution_rem} and maintain the high-quality inversion, we set $\xi=0.4$ and $L=256$ in all the following examples. Moreover, the phenomenon in this experiment also holds for the noisy case.

\subsubsection{The influence of the wavenumber $k$}

\begin{figure}[t]
	\centering
	\subfigure[$\mathrm{Im}(C_{jm}),k=4\pi$]{
		\includegraphics[width=0.3\textwidth]{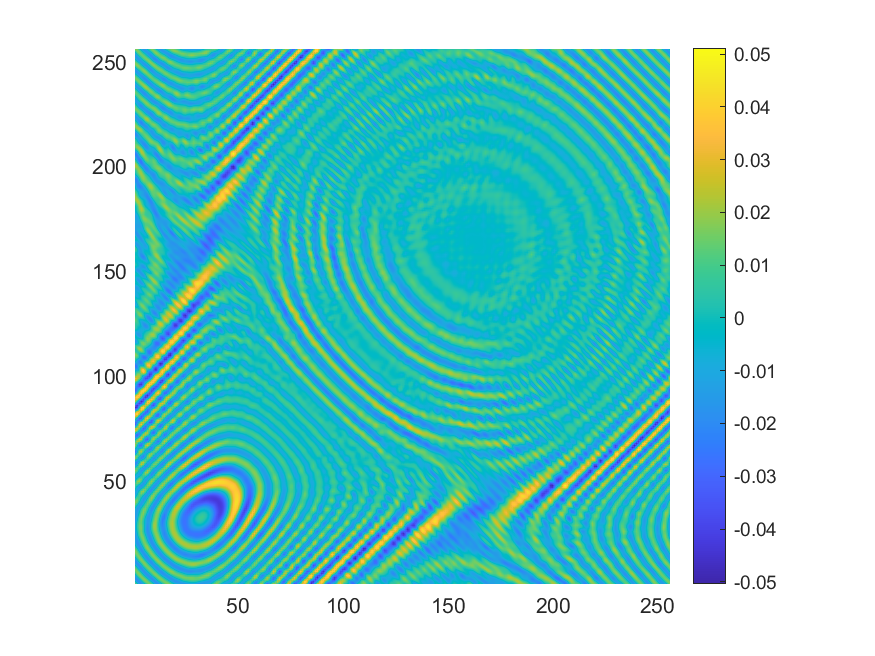}
	}
	\subfigure[$\mathrm{Im}(C_{jm}),k=4\pi$]{
		\includegraphics[width=0.3\textwidth]{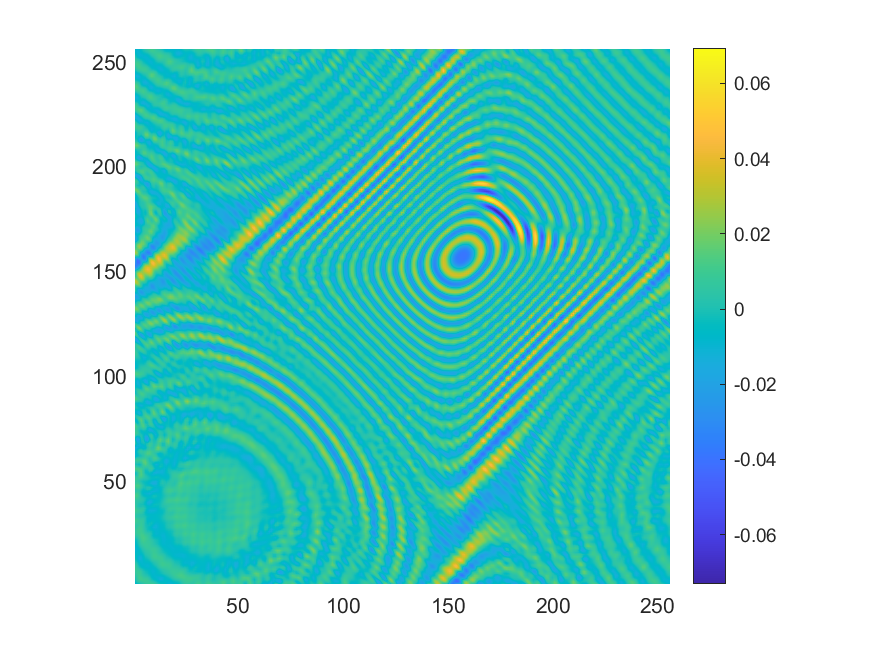}
	}\\
	\subfigure[$\mathrm{Im}(N_{jm}^{s}),k=4\pi$]{
		\includegraphics[width=0.3\textwidth]{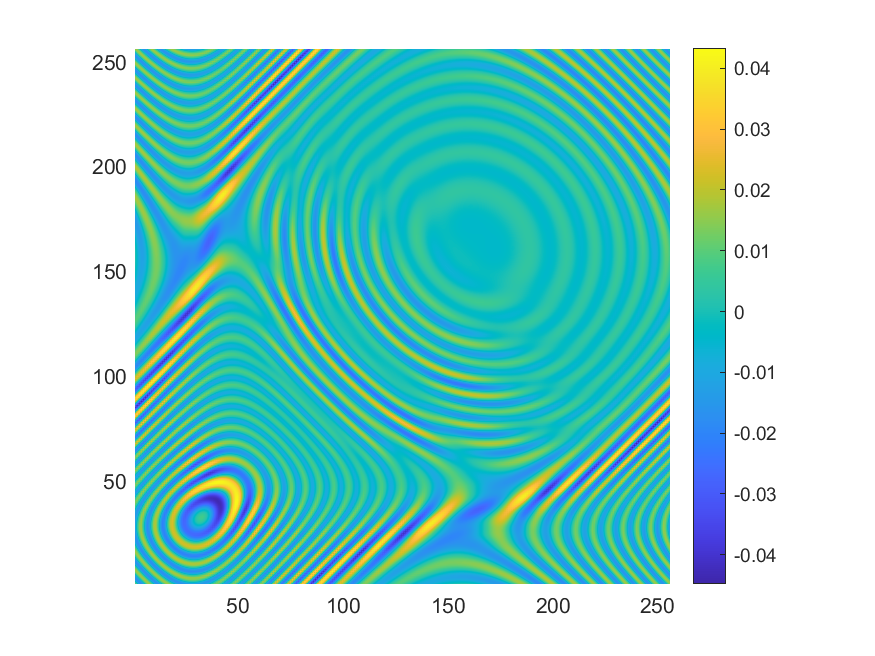}
	}
	\subfigure[$\mathrm{Im}(N_{jm}^{s}),k=4\pi$]{
		\includegraphics[width=0.3\textwidth]{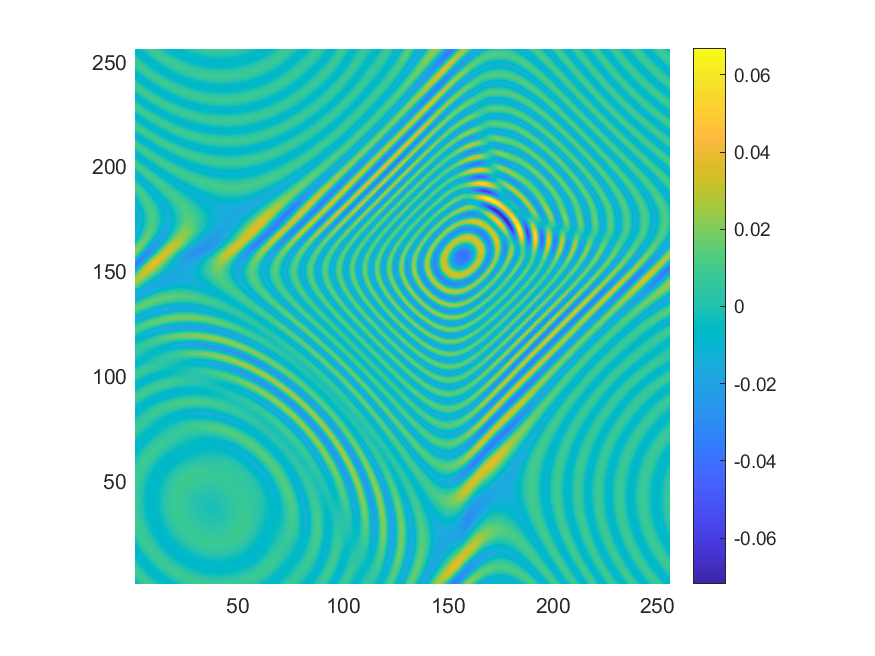}
	}\\
	\subfigure[$\mathrm{Im}(C_{jm}), k=8\pi$]{
		\includegraphics[width=0.3\textwidth]{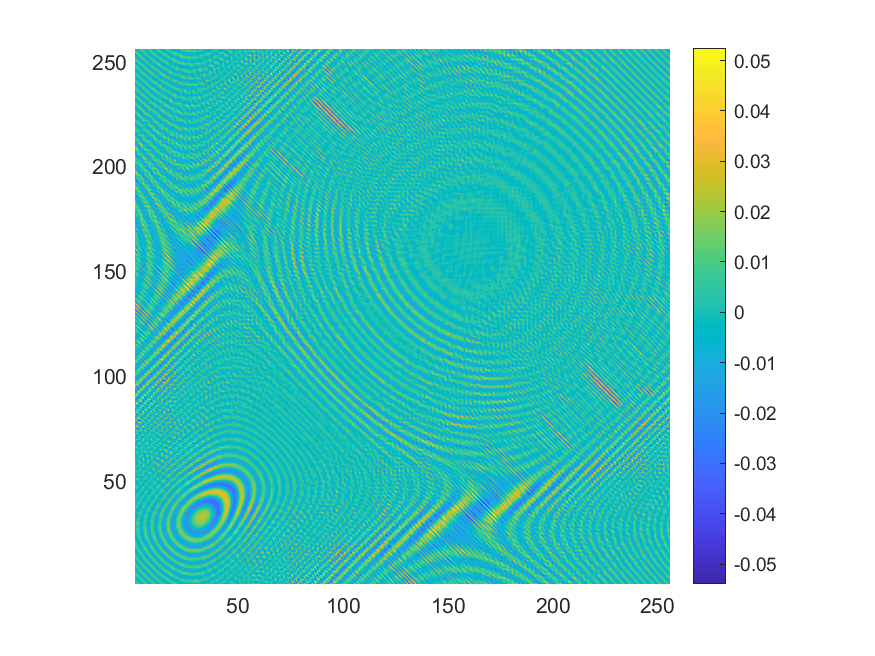}
	}
	\subfigure[$\mathrm{Im}(C_{jm}), k=8\pi$]{
		\includegraphics[width=0.3\textwidth]{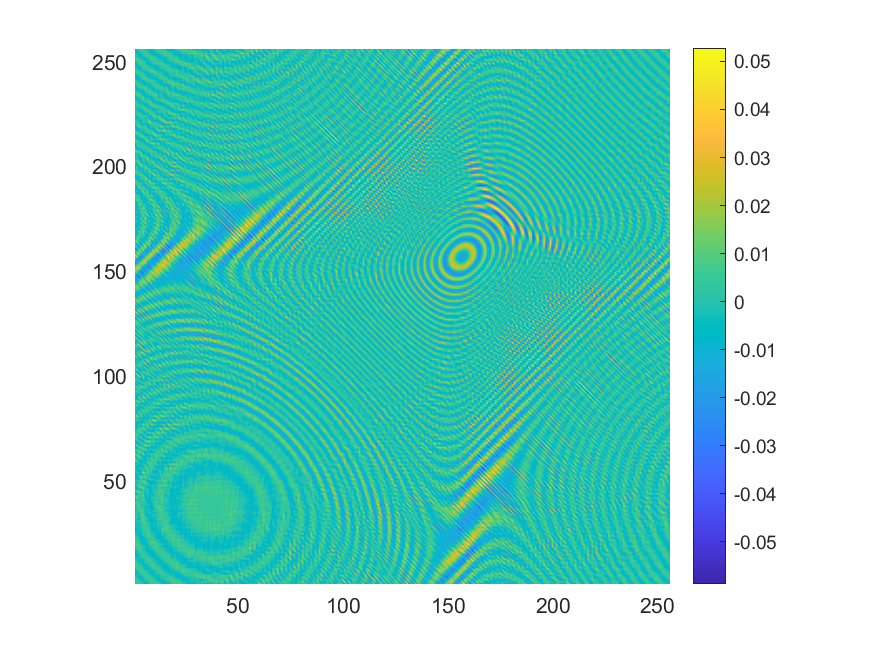}
	}\\
	\subfigure[$\mathrm{Im}(N_{jm}^{s}), k=8\pi$]{
		\includegraphics[width=0.3\textwidth]{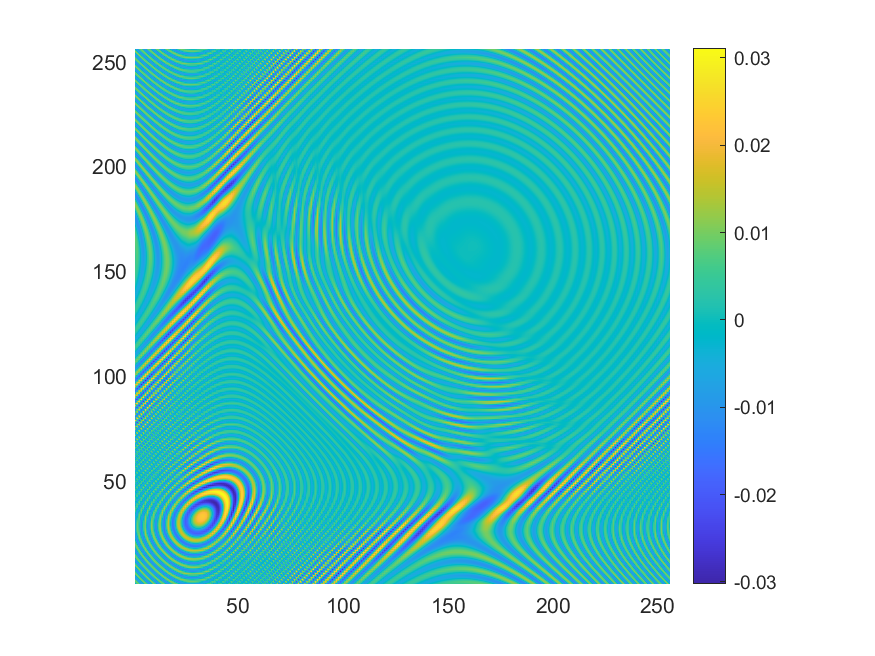}
	}
	\subfigure[$\mathrm{Im}(N_{jm}^{s}), k=8\pi$]{
		\includegraphics[width=0.3\textwidth]{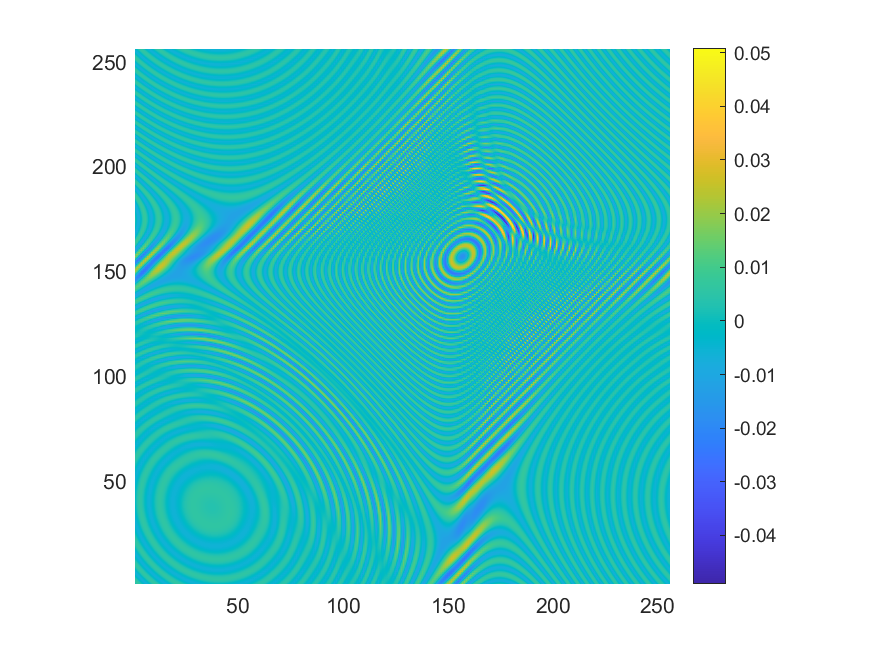}
	}
	\caption{\label{fig:ex2_Cjm}The imaginary parts of $C_{jm}$ and $N^{s}_{jm}$ with $k=4\pi$ and $k=8\pi$ for the kite (Left) and the peanut (Right). The first row displays the imaginary parts of $C_{jm}$ with $k=4\pi$, the second row displays the imaginary parts of $N^{s}_{jm}$ with $k=4\pi$, the third row displays the imaginary parts of $C_{jm}$ with $k=8\pi$ and the fourth row displays the imaginary parts of $N^{s}_{jm}$ with $k=8\pi$.}
\end{figure}

\begin{figure}[h]
	\centering
	\subfigure[Exact shape]{
		\includegraphics[width=0.3\textwidth]{exact_L256_kite.eps}
	}
	\subfigure[$k=4\pi$]{
		\includegraphics[width=0.3\textwidth]{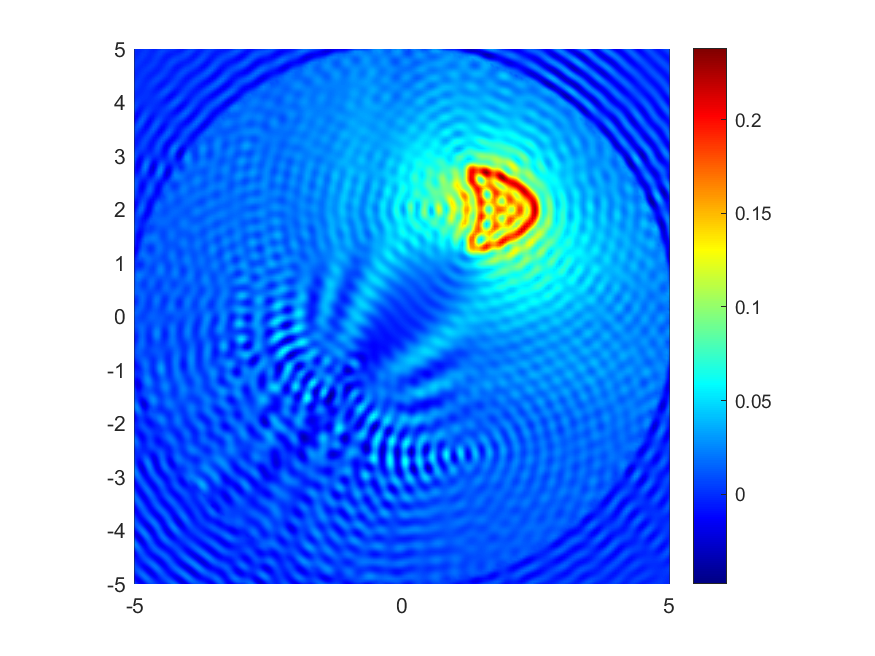}
	}
	\subfigure[$k=8\pi$]{
		\includegraphics[width=0.3\textwidth]{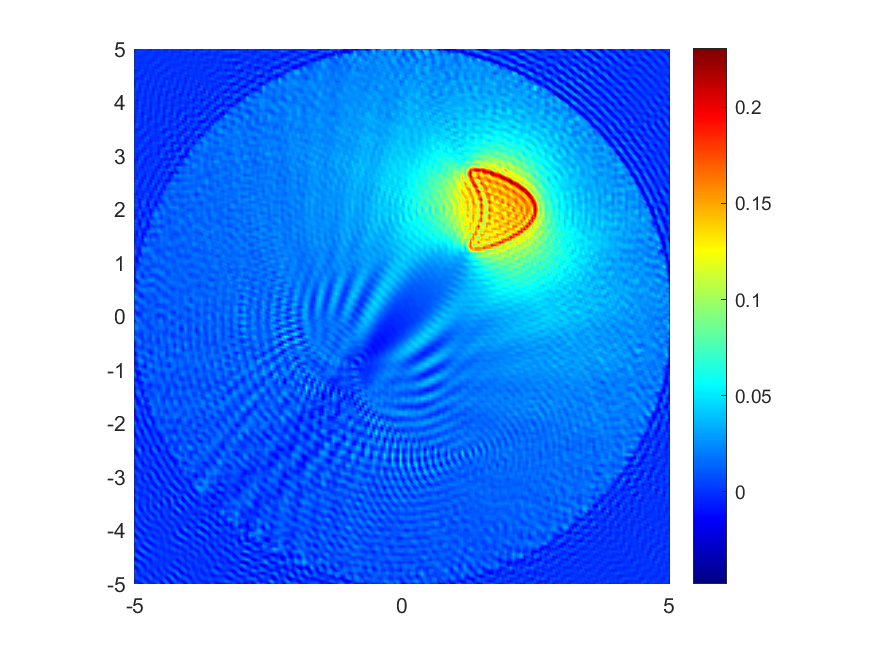}
	}\\
	\subfigure[Exact shape]{
		\includegraphics[width=0.3\textwidth]{exact_L256_peanut.eps}
	}
	\subfigure[$k=4\pi$]{
		\includegraphics[width=0.3\textwidth]{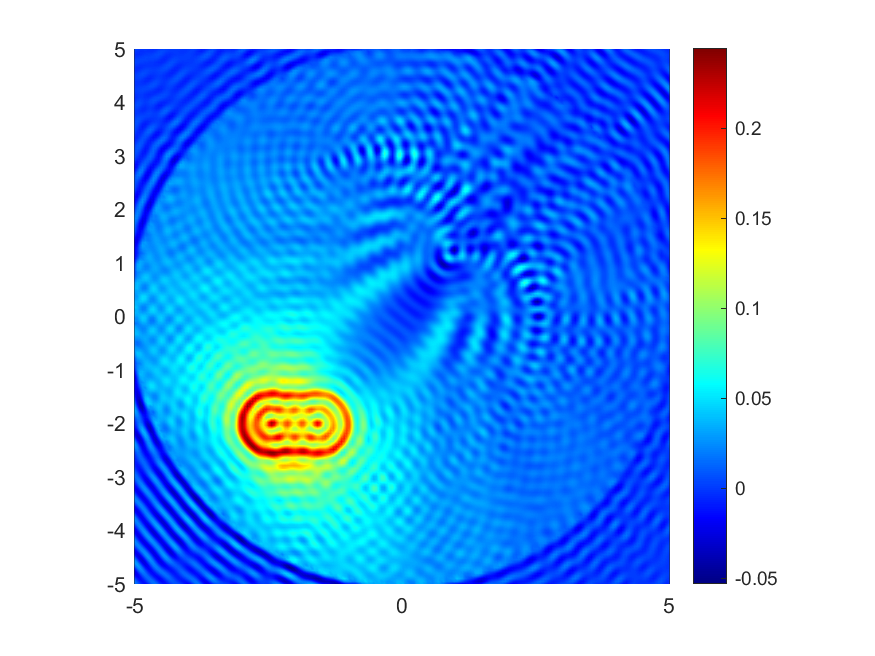}
	}
	\subfigure[$k=8\pi$]{
		\includegraphics[width=0.3\textwidth]{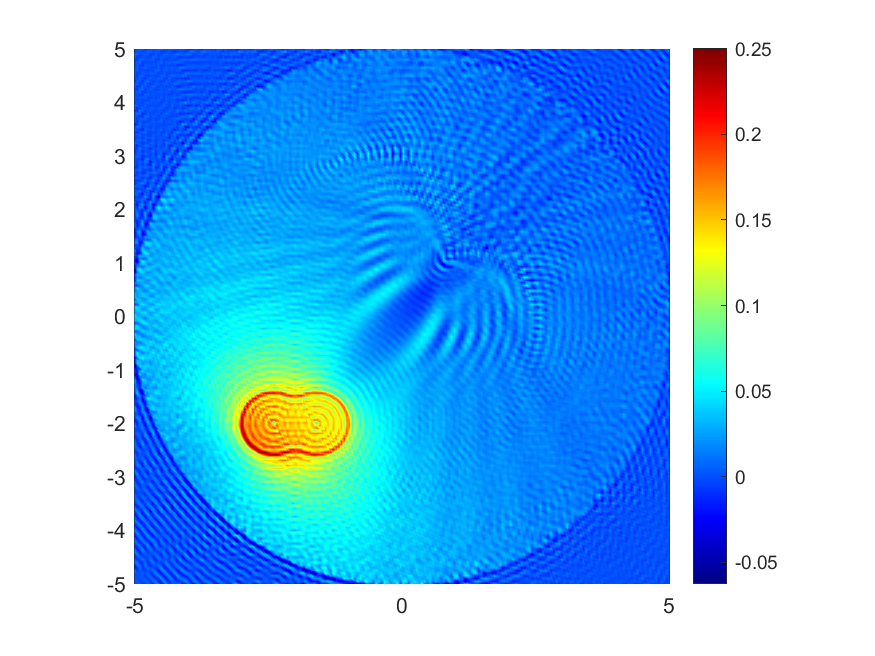}
	}
    \caption{\label{fig:ex2_curve}Recoveries of $\partial D$ by DCM with $k=4\pi$ and $k=8\pi$ for the kite (Top) and the peanut (Bottom). The first column shows the exact boundary $\partial D$ and $L=256$ measurement points, the second column shows the reconstructed $\partial D$ with $k=4\pi$ and the third column shows the reconstructed $\partial D$ with $k=8\pi$.}
\end{figure}

The obstacle reconstructions with different $k$ are considered by using noise-free data, i.e., $\delta=0$. We set $k=4\pi$ and $k=8\pi$. The imaginary parts of the cross-correlations $C_{jm}$ with $k=4\pi$ and $k=8\pi$ for the kite-shaped obstacle and the peanut-shaped obstacle are plotted in the first and third rows of Fig. \ref{fig:ex2_Cjm}. The approximate active imaginary parts of $N^{s}_{jm}$ are shown in the second and fourth rows of Fig. \ref{fig:ex2_Cjm}, respectively. The results of Fig. \ref{fig:ex2_Cjm} also confirm the formula \eqref{eq:app_c_scattered} that describes the relationship between $C_{jm}$ and $N^{s}_{jm}$. We present the boundary reconstructions of $\partial D$ with $k=4\pi$ and $k=8\pi$ for the kite and the peanut in Fig. \ref{fig:ex2_curve}.

It is noted in Fig. \ref{fig:ex2_curve} that with the increase of the wavenumber $k$, more detailed geometrical information can be attained, which can significantly improve inversion quality. This phenomenon occurs mainly because as the wavenumber $k$ increases, smaller probe wavelength can be used and $k\int_{S}|\psi^{\infty}(\hat{x},\tau)|^{2}\mathrm{d}\hat{x}$ in Theorem \ref{thm:ind_resolution_rem} is accordingly increased. Furthermore, compared with the reconstructions of the kite-shaped scatterer in \cite{GHM1} for high wavenumber, reconstructions obtained by our DCM are more accurate.

\subsubsection{The influence of the noise level $\delta$}

\begin{figure}[h]
	\centering
	\subfigure[$\mathrm{Im}(C^{\delta}_{jm}),\delta=0.2$]{
		\includegraphics[width=0.3\textwidth]{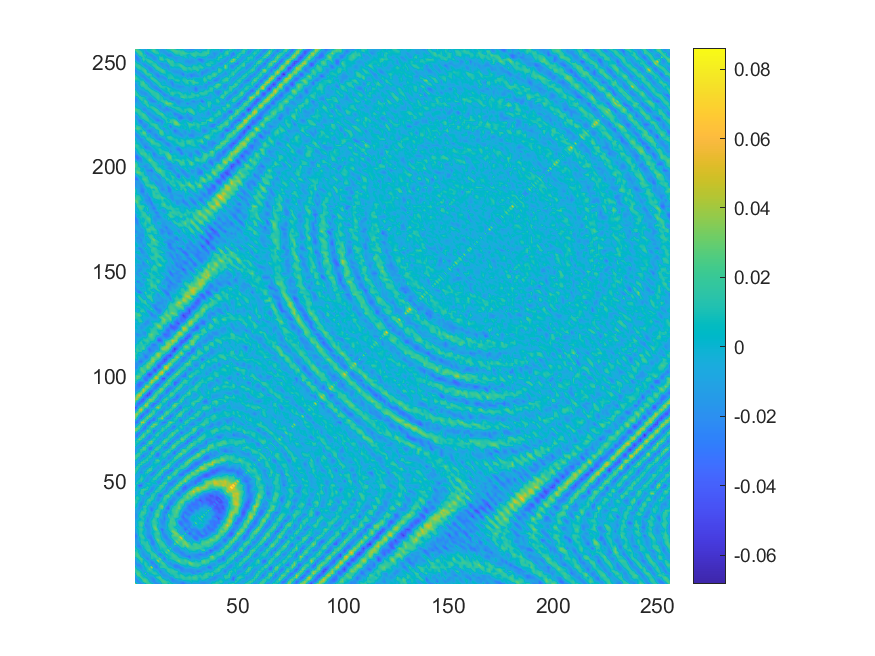}
	}
	\subfigure[$\mathrm{Im}(C^{\delta}_{jm}),\delta=0.4$]{
		\includegraphics[width=0.3\textwidth]{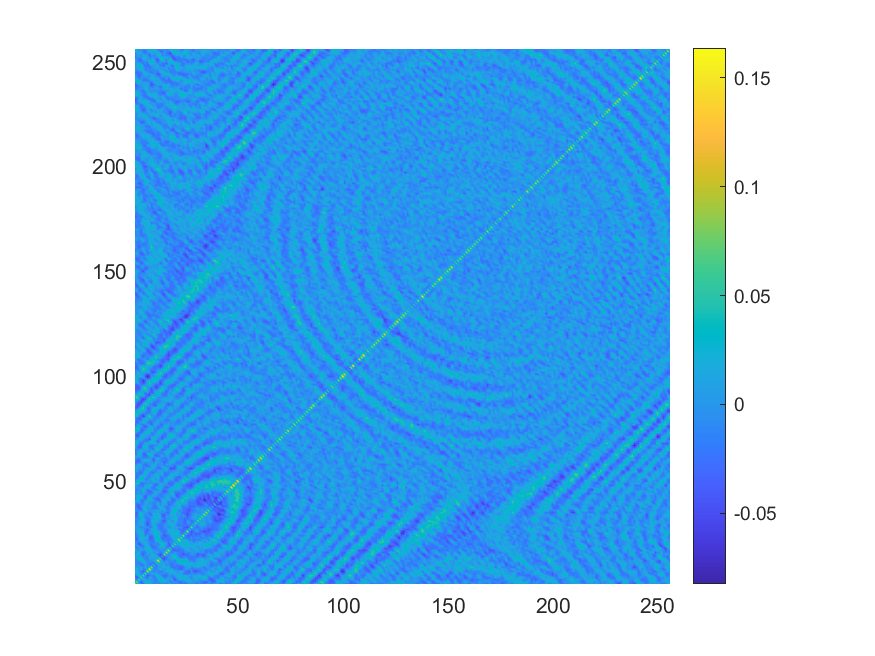}
	}
	\subfigure[$\mathrm{Im}(N_{jm}^{s})$]{
		\includegraphics[width=0.3\textwidth]{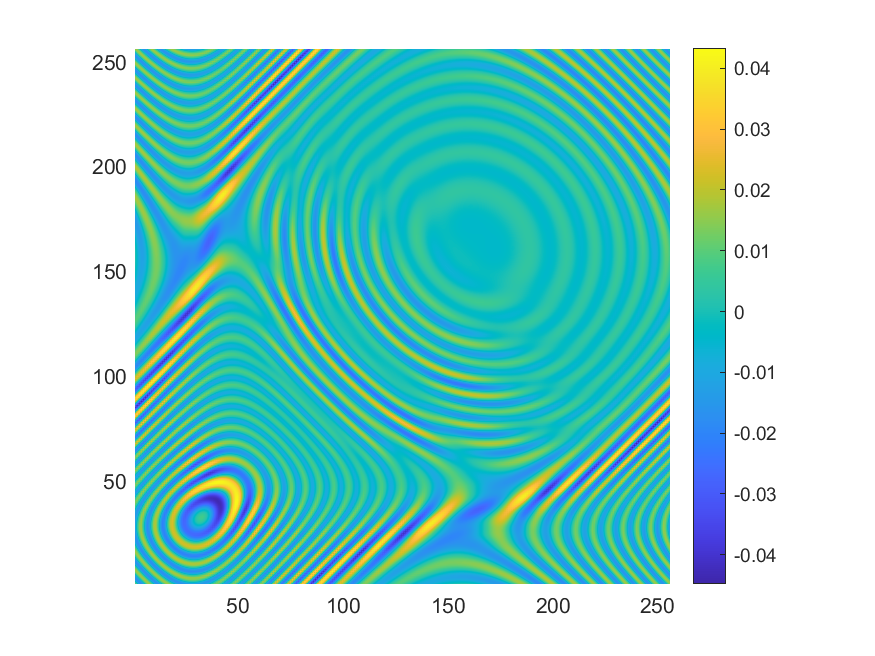}
	}\\
	\subfigure[$\mathrm{Im}(C^{\delta}_{jm}),\delta=0.2$]{
		\includegraphics[width=0.3\textwidth]{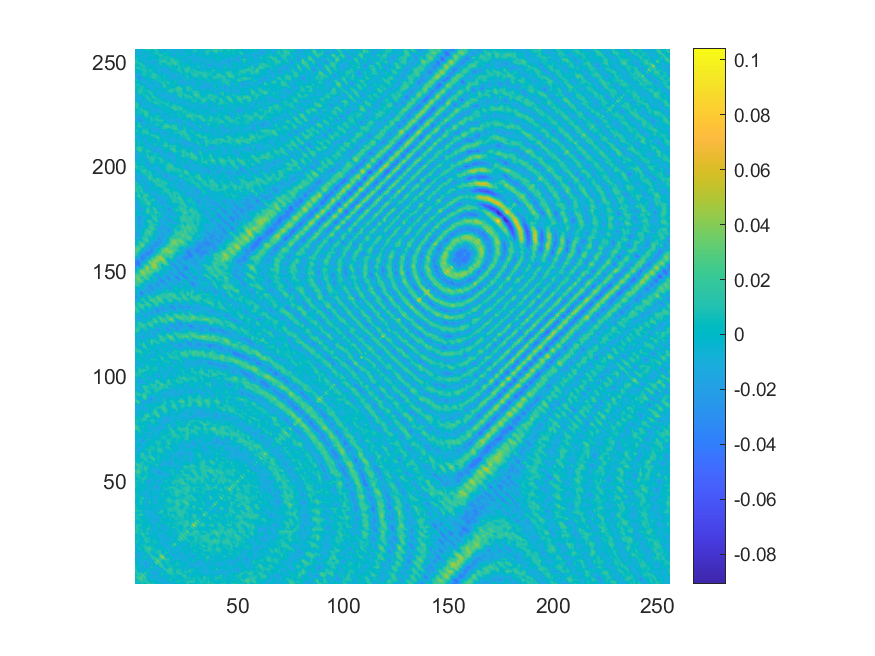}
	}
	\subfigure[$\mathrm{Im}(C^{\delta}_{jm}),\delta=0.4$]{
		\includegraphics[width=0.3\textwidth]{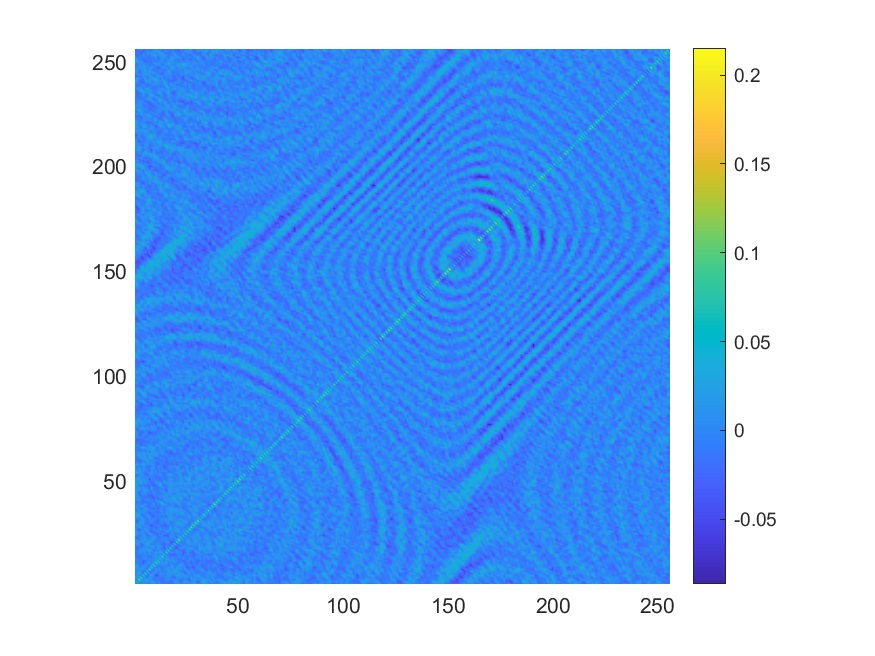}
	}
	\subfigure[$\mathrm{Im}(N_{jm}^{s})$]{
		\includegraphics[width=0.3\textwidth]{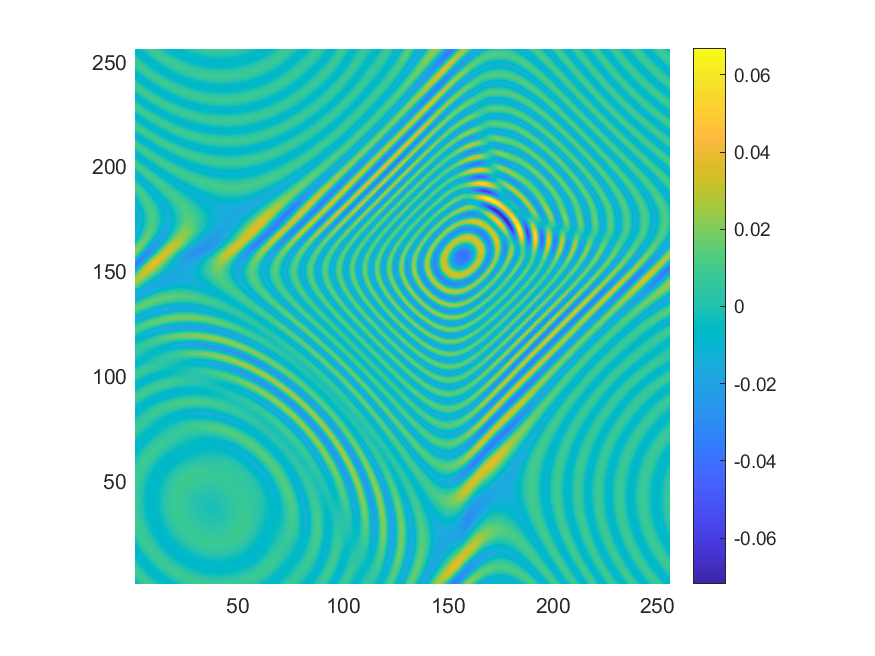}
	}
	\caption{\label{fig:ex3_Cjm}The imaginary parts of $C^{\delta}_{jm}$ and $N^{s}_{jm}$ with $\delta=0.2$ and $\delta=0.4$ for the kite (Top) and the peanut (Bottom). The first column displays the imaginary parts of $C^{\delta}_{jm}$ with $\delta=0.2$, the second column displays the imaginary parts of $C^{\delta}_{jm}$ with $\delta=0.4$ and the third column displays the imaginary parts of $N^{s}_{jm}$.}
\end{figure}

\begin{figure}[h]
	\centering
	\subfigure[Exact shape]{
		\includegraphics[width=0.3\textwidth]{exact_L256_kite.eps}
	}
	\subfigure[$\delta=0.2$]{
		\includegraphics[width=0.3\textwidth]{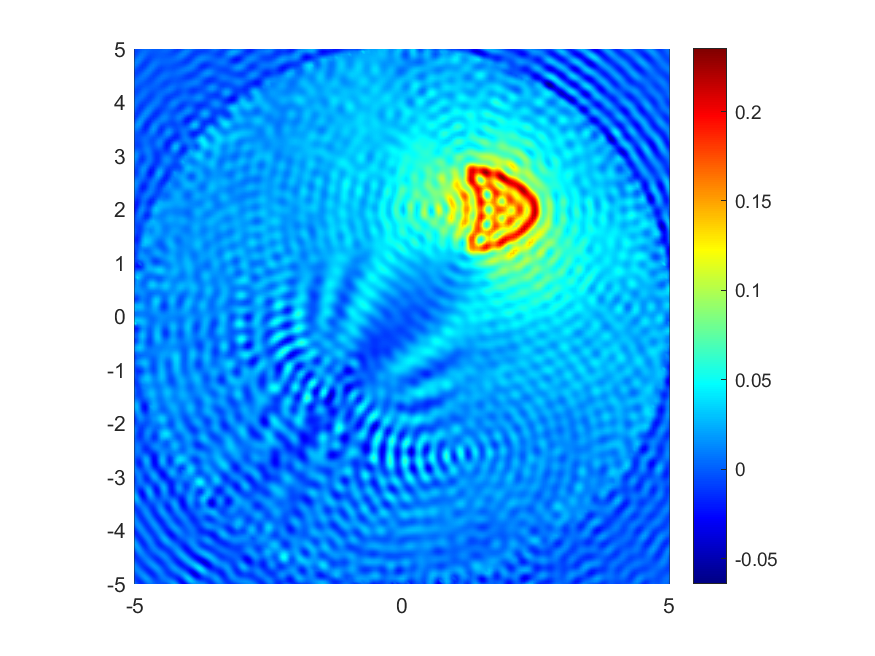}
	}
	\subfigure[$\delta=0.4$]{
		\includegraphics[width=0.3\textwidth]{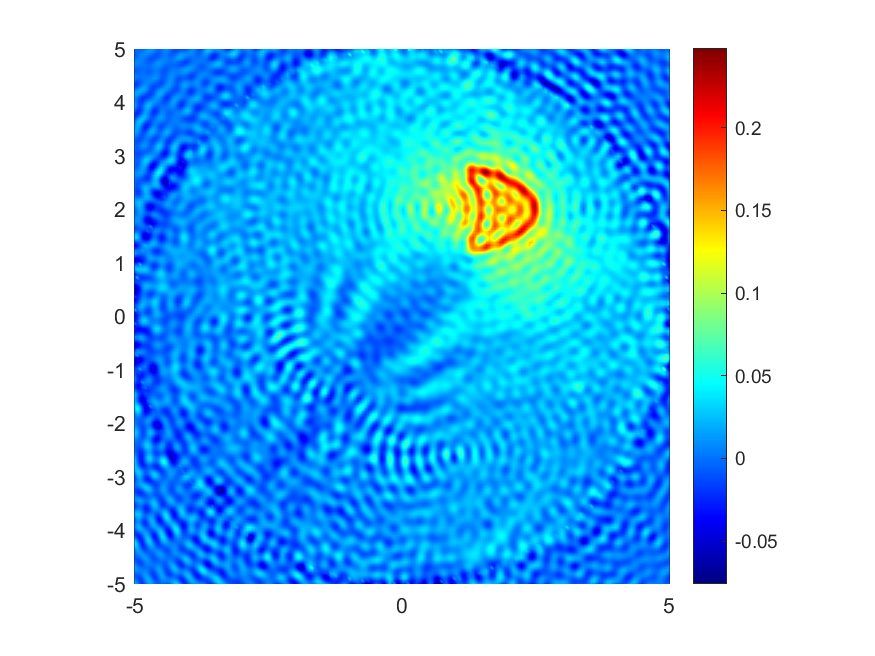}
	}\\
	\subfigure[Exact shape]{
		\includegraphics[width=0.3\textwidth]{exact_L256_peanut.eps}
	}
	\subfigure[$\delta=0.2$]{
		\includegraphics[width=0.3\textwidth]{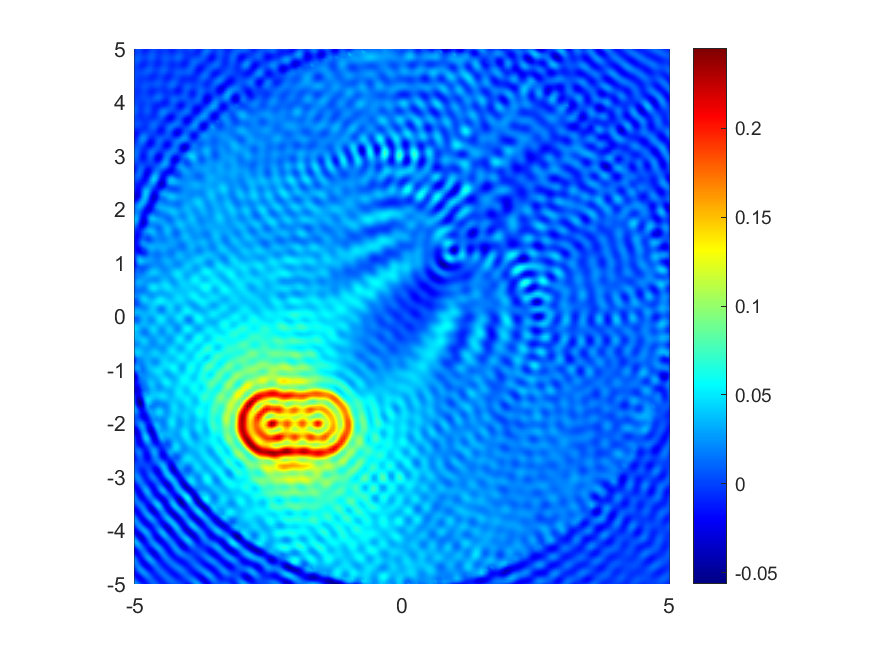}
	}
	\subfigure[$\delta=0.4$]{
		\includegraphics[width=0.3\textwidth]{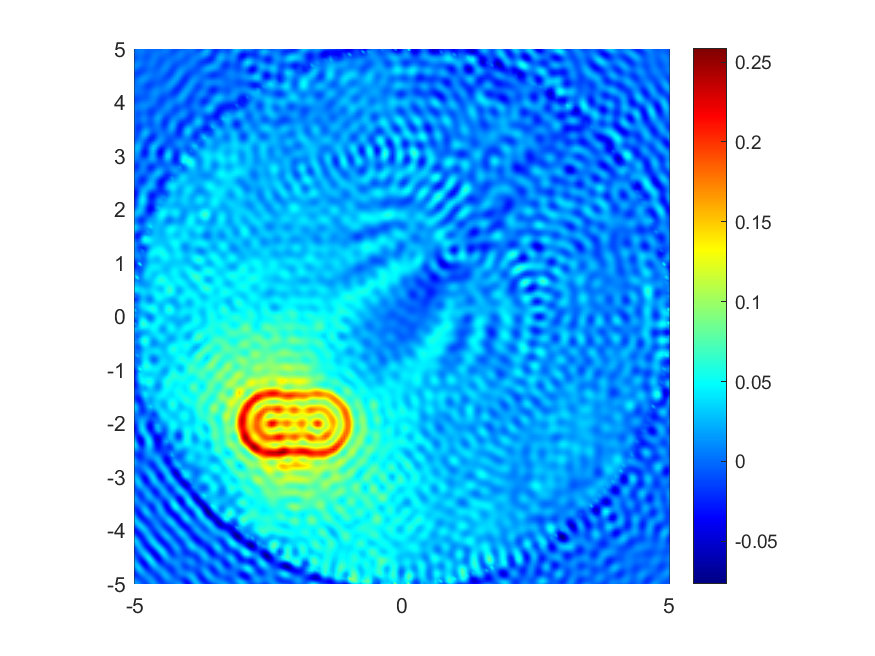}
	}
    \caption{\label{fig:ex3_curve}Recoveries of $\partial D$ by DCM with $\delta=0.2$ and $\delta=0.4$ for the kite (Top) and the peanut (Bottom). The first column shows the exact boundary $\partial D$ and $L=256$ measurement points, the second column shows the reconstructed $\partial D$ with $\delta=0.2$ and the third column shows the reconstructed $\partial D$ with $\delta=0.4$.}
\end{figure}

The above two experiments are only conducted in the noise-free case. However, the observation data is typically perturbed by the noise due to the limitation of many practical scenarios. Hence, in this example, we shall consider the influence of the noise level $\delta$. To do so, we set $\delta=0.2$ and $\delta=0.4$, respectively. The wavenumber is $k=4\pi$. We present the imaginary parts of the cross-correlations $C^{\delta}_{jm}$ with $\delta=0.2$ and $\delta=0.4$ for the kite-shaped obstacle in Fig. \ref{fig:ex3_Cjm}(a) and (b), and the imaginary parts of the cross-correlations for the peanut-shaped obstacle are shown in Fig. \ref{fig:ex3_Cjm}(d) and (e). The approximate active imaginary parts of $N^{s}_{jm}$ for the kite and the peanut can be seen in Fig. \ref{fig:ex3_Cjm}(c) and (f). Clearly, the error between $C^{\delta}_{jm}$ and $N^{s}_{jm}$ will become larger when the noise level $\delta$ increases, which is reasonable since in the noisy case the error between $C^{\delta}_{jm}$ and $N^{s}_{jm}$ is generated by not only the quadrature error of the trapezoidal rule approximation but also the added noise. The recovery results of the boundary $\partial D$ are shown in Fig. \ref{fig:ex3_curve}, from which we clearly observe that our proposed DCM can obtain satisfactory reconstructions even for the case of $\delta=0.4$. In fact, this is expected since Theorem \ref{thm:stability_rem} has theoretically indicated that our proposed DCM is extremely stable with respect to the added noise.

\subsection{Example 2: Multiple obstacles case}
In this example, we aim to investigate the performance of our DCM in the more complicated case of multiple obstacles. The noise level is chosen as $\delta=0.2$.

\subsubsection{Tow obstacles close to each other}

\begin{figure}[h]
	\centering
	\subfigure[$\mathrm{Im}(C^{\delta}_{jm}),k=2\pi$]{
		\includegraphics[width=0.3\textwidth]{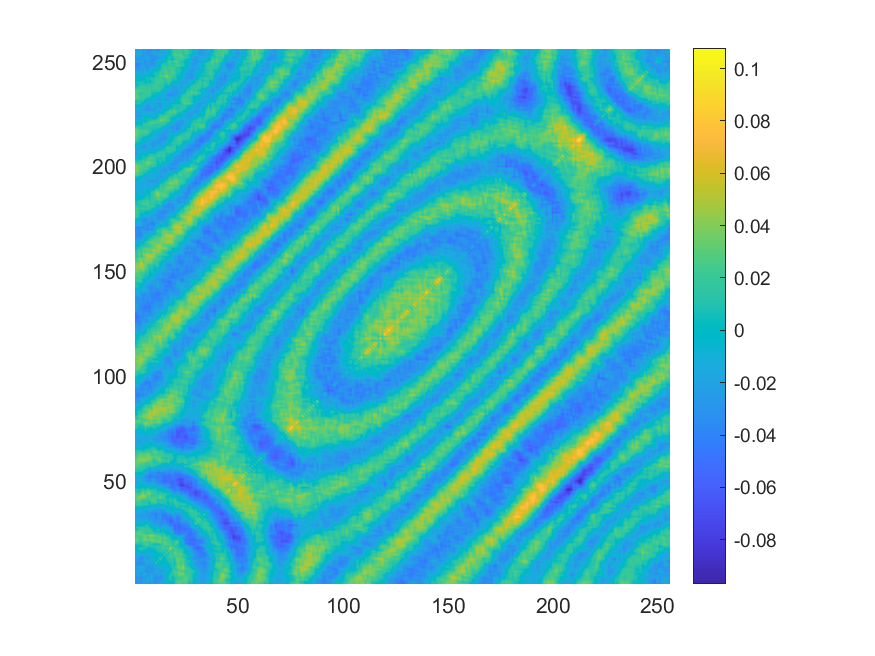}
	}
	\subfigure[$\mathrm{Im}(N_{jm}^{s}),k=2\pi$]{
		\includegraphics[width=0.3\textwidth]{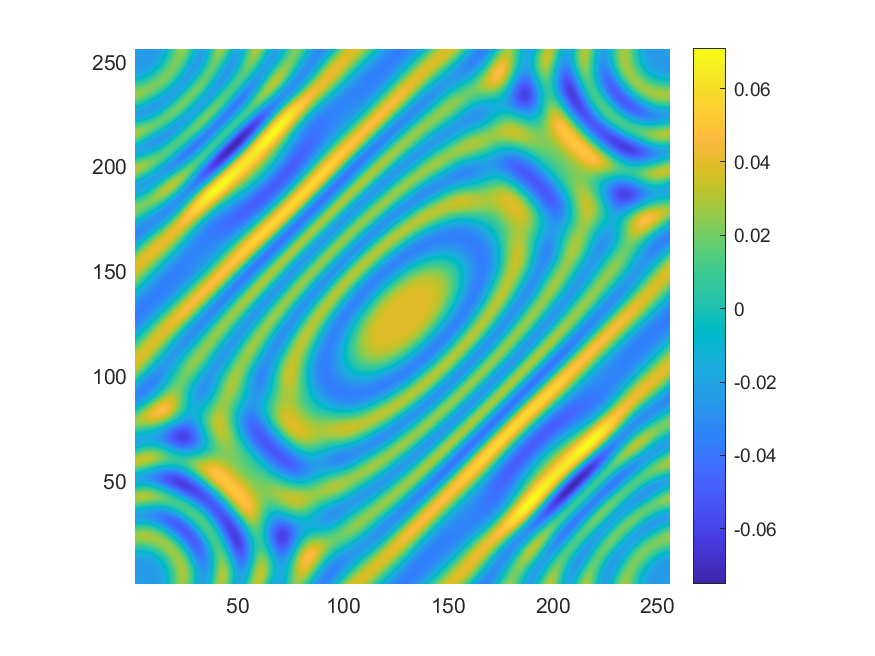}
	}\\
	\subfigure[$\mathrm{Im}(C^{\delta}_{jm}), k=4\pi$]{
		\includegraphics[width=0.3\textwidth]{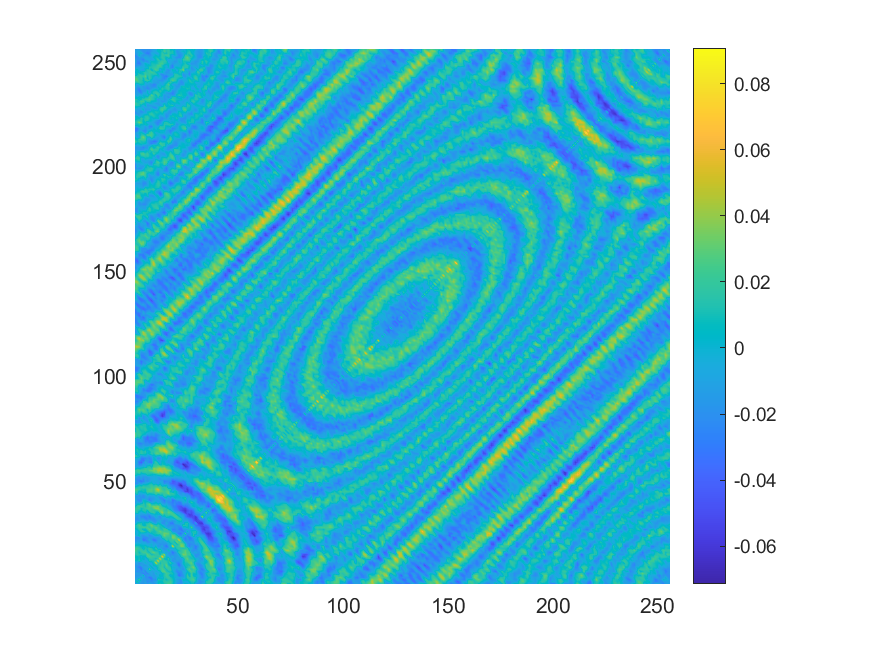}
	}
	\subfigure[$\mathrm{Im}(N_{jm}^{s}), k=4\pi$]{
		\includegraphics[width=0.3\textwidth]{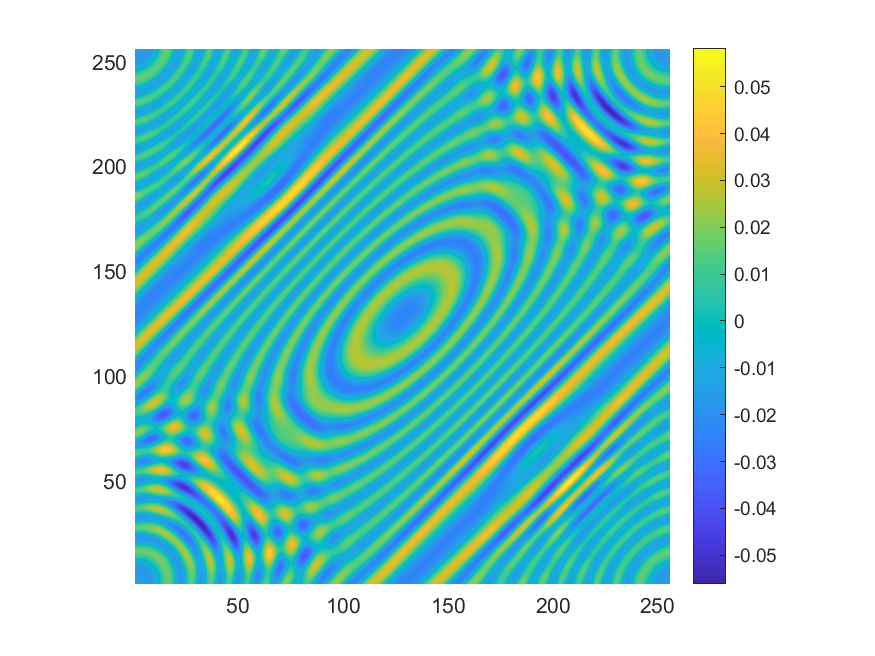}
	}
	\caption{\label{fig:ex4_limit_Cjm}The imaginary parts of $C^{\delta}_{jm}$ and $N^{s}_{jm}$ with $k=2\pi$ (Top) and $k=4\pi$ (Bottom) for the resolution limit case. The first column displays the imaginary parts of $C^{\delta}_{jm}$ with $k=2\pi$ and $k=4\pi$, and the second column displays the imaginary parts of $N^{s}_{jm}$ with $k=2\pi$ and $k=4\pi$.}
\end{figure}

\begin{figure}[h]
	\centering
	\subfigure[Exact shape]{
		\includegraphics[width=0.3\textwidth]{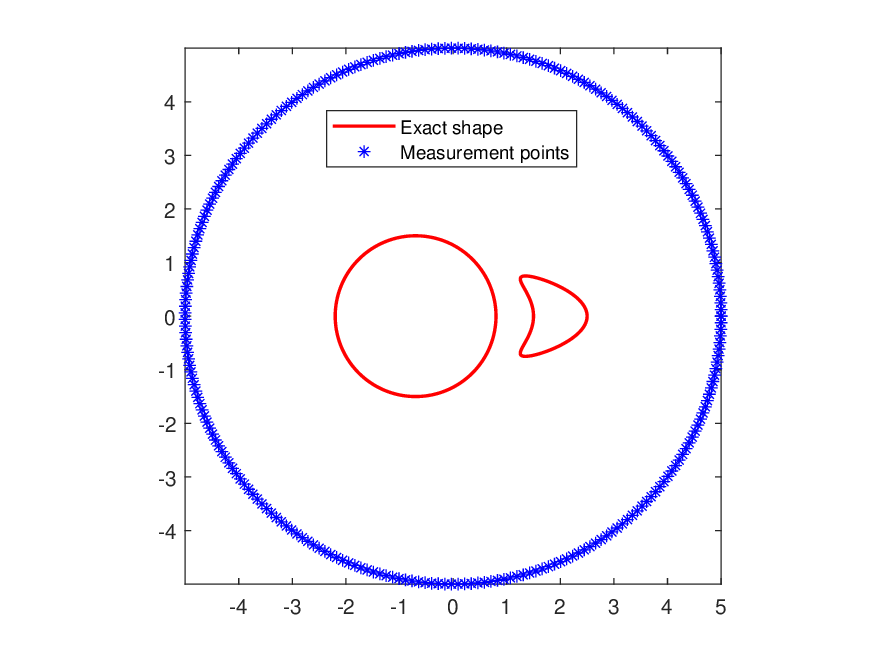}
	}
	\subfigure[$k=2\pi$]{
		\includegraphics[width=0.3\textwidth]{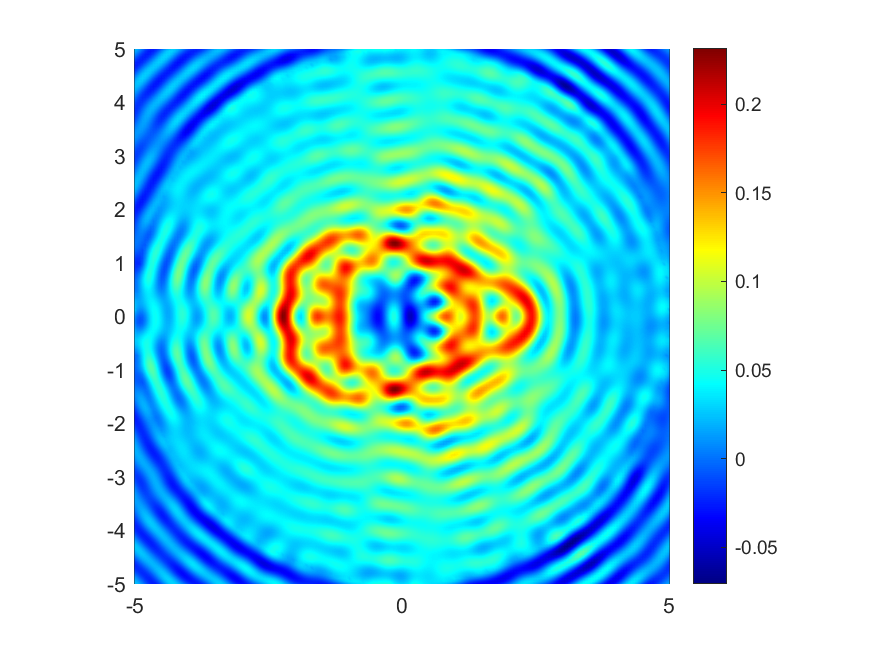}
	}
	\subfigure[$k=4\pi$]{
		\includegraphics[width=0.3\textwidth]{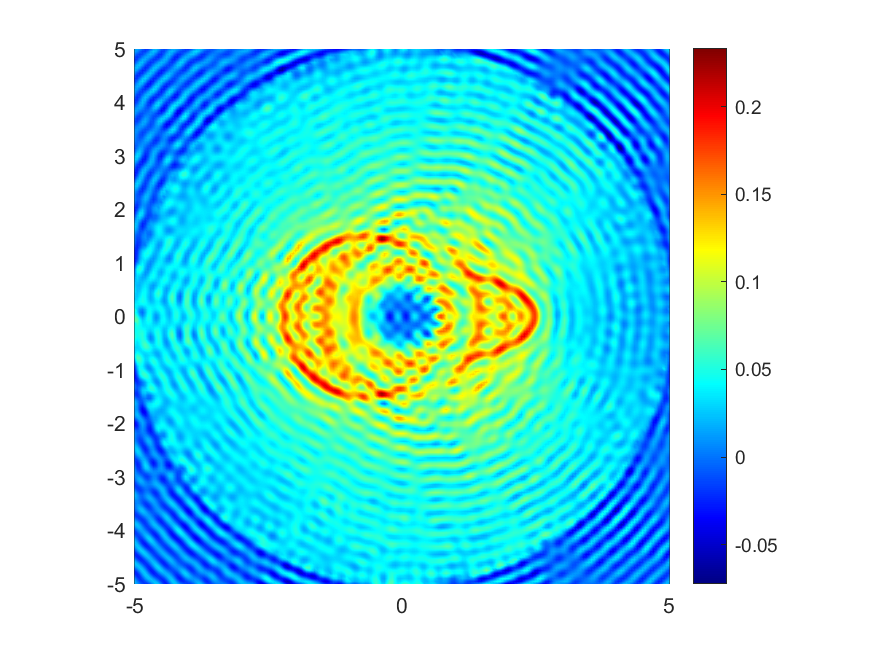}
	}
    \caption{\label{fig:ex4_limit_curve}Recoveries of $\partial D$ by DCM with $k=2\pi$ and $k=4\pi$ for the resolution limit. The first column shows the exact boundary $\partial D$ and $L=256$ measurement points, the second column shows the reconstructed $\partial D$ with $k=2\pi$ and the third column shows the reconstructed $\partial D$ with $k=4\pi$.}
\end{figure}

We consider the imaging of two scatterers close to each other in this experiment. Thus, the considered obstacle $D$ is the union of two bounded simply connected domains, i.e., $D=D_{1}\bigcup D_{2}$. Moreover, these two scatterers are disjoint. Here, the parameterized forms of $D_{1}$ and $D_{2}$ are given by
\begin{equation}
\begin{aligned}
\partial D_{1} = \bigg\{(2+0.5(\cos \theta + 0.65\cos 2\theta - 0.65), 0.75\sin \theta):\theta\in [0,2\pi]\bigg\},
\nonumber
\end{aligned}
\end{equation}
and
\begin{equation}
\begin{aligned}
\partial D_{2} = \bigg\{(-0.7+1.5\cos \theta, 1.5\sin \theta):\theta\in [0,2\pi]\bigg\}.
\nonumber
\end{aligned}
\end{equation}
Indeed, these two obstacles are very close to each other and the distance two obstacles is about 0.5. Therefore, we set $k=2\pi$ and $k=4\pi$ for imaging $D$ and the corresponding probe wavelength defined by $\frac{2\pi}{k}$ are respectively 1 and 0.5. The imaginary parts of the cross-correlations $C^{\delta}_{jm}$ for $k=2\pi$ and $k=4\pi$ are plotted in Fig. \ref{fig:ex4_limit_Cjm}(a) and (c), respectively. The approximate active imaginary parts of $N^{s}_{jm}$ are respectively shown in Fig. \ref{fig:ex4_limit_Cjm}(b) and (d). Fig. \ref{fig:ex4_limit_Cjm} indicates that even for two obstacles close to each other, the approximation relationship between $C^{\delta}_{jm}$ and $N^{s}_{jm}$ also holds. The reconstructions of $\partial D$ in the resolution limit case with $k=2\pi$ and $k=4\pi$ are presented in Fig. \ref{fig:ex4_limit_curve}. In the $k=2\pi$ case, the gap between these two obstacles is about half a wavelength. However, we unfortunately find from Fig. \ref{fig:ex4_limit_curve}(b) that for the wavenumber $k=2\pi$ the considered two obstacles can not be well distinguished. When the wavenumber is increased to $k=4\pi$, the gap between two obstacles can be clearly depicted.

\subsubsection{Multiscale}
We now turn our attention to the case of multiscale. To this end, we set $k=4\pi$ and $k=8\pi$, respectively. Furthermore, the considered obstacle $D=D_{1}\bigcup D_{2}$ is the union of a big pear-shaped obstacle $D_{1}$ located at $(0,0)$ and a small disk centered at $(2,3)$ with the radius 0.2. Here, the parameterized forms of $D_{1}$ and $D_{2}$ are given by
\begin{equation}
\begin{aligned}
\partial D_{1} = \bigg\{(2+0.3\cos 3\theta)(\cos \theta, \sin \theta):\theta\in [0,2\pi]\bigg\},
\nonumber
\end{aligned}
\end{equation}
and
\begin{equation}
\begin{aligned}
\partial D_{2} = \bigg\{(2+0.2\cos \theta, 3+0.2\sin \theta):\theta\in [0,2\pi]\bigg\}.
\nonumber
\end{aligned}
\end{equation}
The imaginary parts of the cross-correlations $C^{\delta}_{jm}$ for $k=4\pi$ and $k=8\pi$ in the multiscale case are plotted in Fig. \ref{fig:ex4_multiscalar_Cjm}(a) and (c), respectively. The approximate active imaginary parts of $N^{s}_{jm}$ are respectively shown in Fig. \ref{fig:ex4_multiscalar_Cjm}(b) and (d). From Fig. \ref{fig:ex4_multiscalar_Cjm} we find that $C^{\delta}_{jm}$ is also the discrete approximation of $N^{s}_{jm}$ for the multiscale case. We show the reconstructions of $\partial D$ in the multiscale case with $k=4\pi$ and $k=8\pi$ in Fig. \ref{fig:ex4_multiscalar_curve}. It is clear in Fig. \ref{fig:ex4_multiscalar_curve} that our proposed DCM is capable of recovering both the big and small obstacles.

\begin{figure}[h]
	\centering
	\subfigure[$\mathrm{Im}(C^{\delta}_{jm}),k=4\pi$]{
		\includegraphics[width=0.3\textwidth]{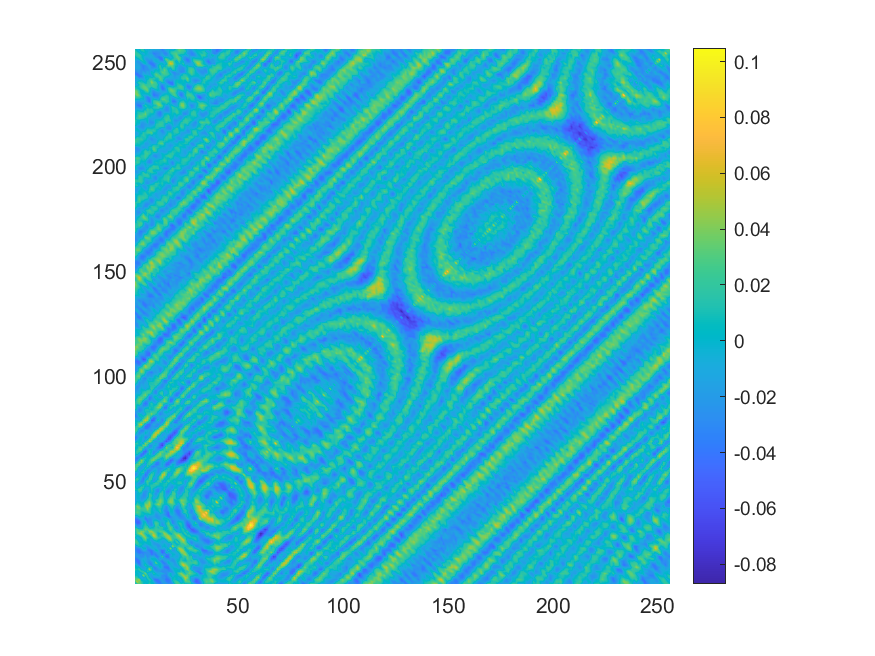}
	}
	\subfigure[$\mathrm{Im}(N_{jm}^{s}),k=4\pi$]{
		\includegraphics[width=0.3\textwidth]{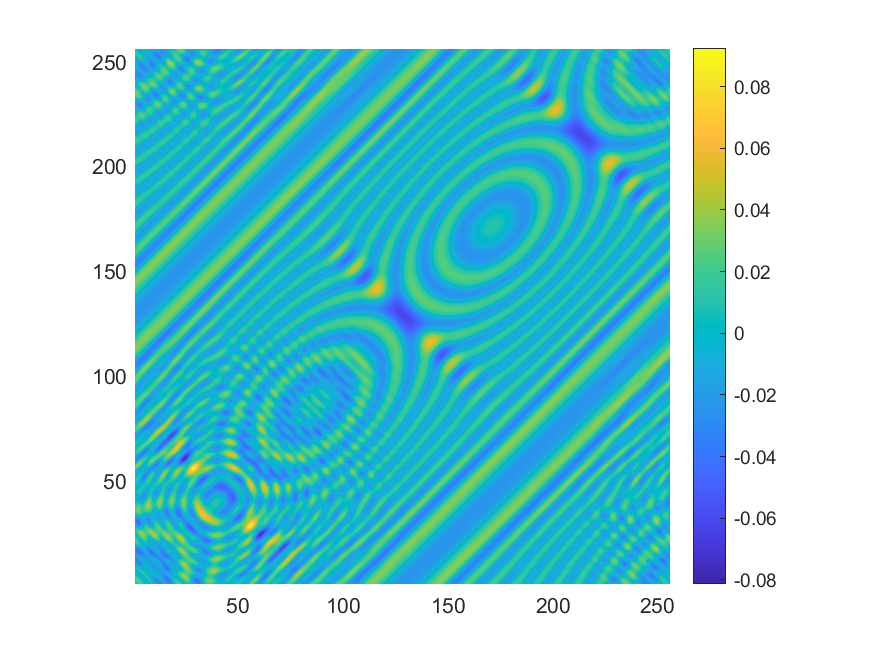}
	}\\
	\subfigure[$\mathrm{Im}(C^{\delta}_{jm}), k=8\pi$]{
		\includegraphics[width=0.3\textwidth]{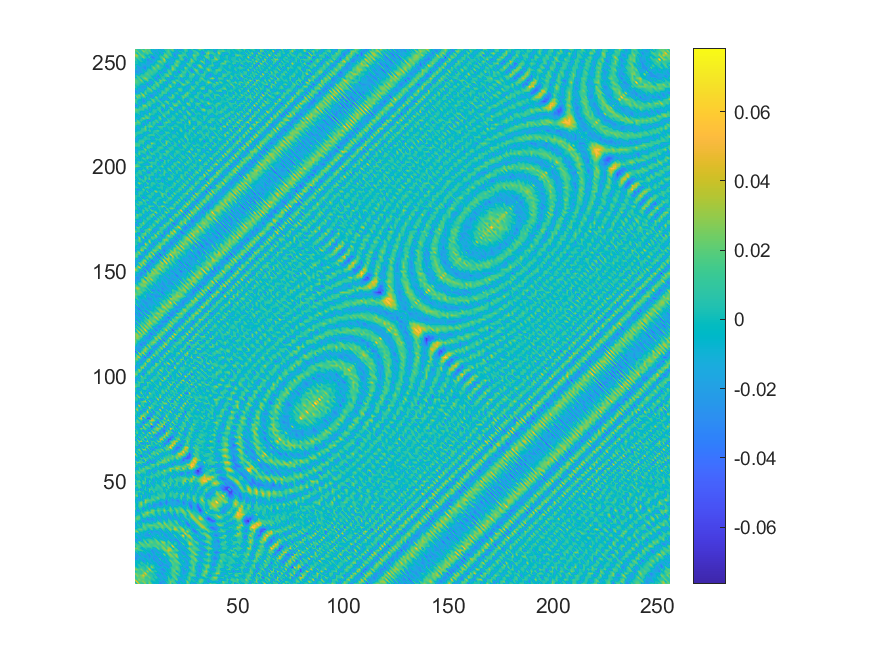}
	}
	\subfigure[$\mathrm{Im}(N_{jm}^{s}), k=8\pi$]{
		\includegraphics[width=0.3\textwidth]{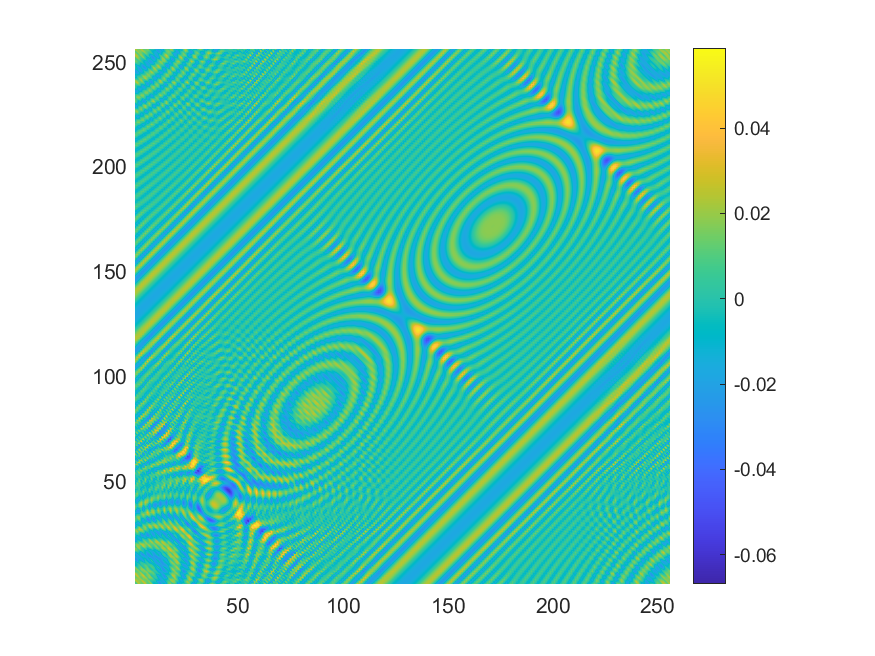}
	}
	\caption{\label{fig:ex4_multiscalar_Cjm}The imaginary parts of $C^{\delta}_{jm}$ and $N^{s}_{jm}$ with $k=4\pi$ (Top) and $k=8\pi$ (Bottom) for the multiscale case. The first column displays the imaginary parts of $C^{\delta}_{jm}$ with $k=4\pi$ and $k=8\pi$, and the second column displays the imaginary parts of $N^{s}_{jm}$ with $k=4\pi$ and $k=8\pi$.}
\end{figure}

\begin{figure}[h]
	\centering
	\subfigure[Exact shape]{
		\includegraphics[width=0.3\textwidth]{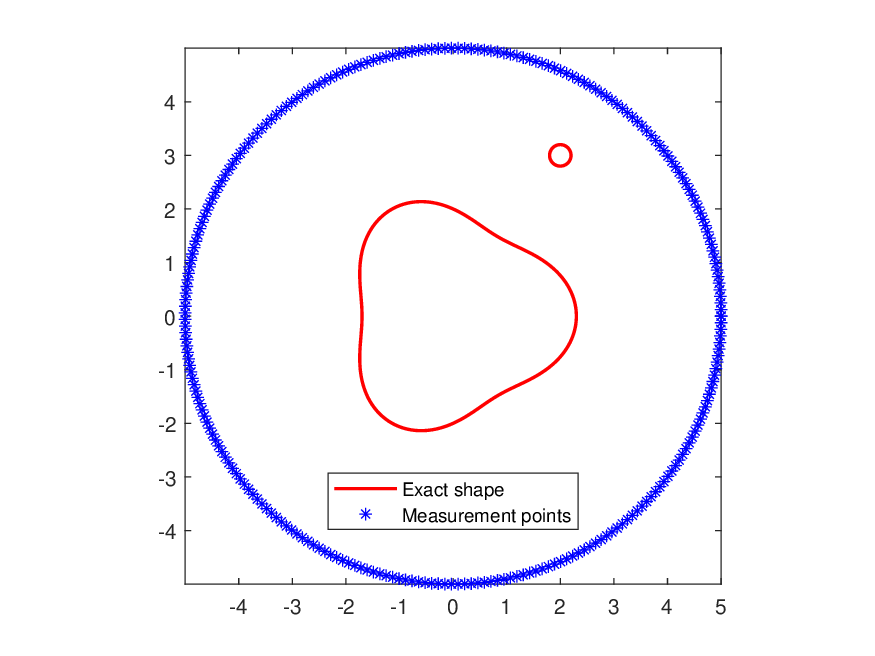}
	}
	\subfigure[$k=4\pi$]{
		\includegraphics[width=0.3\textwidth]{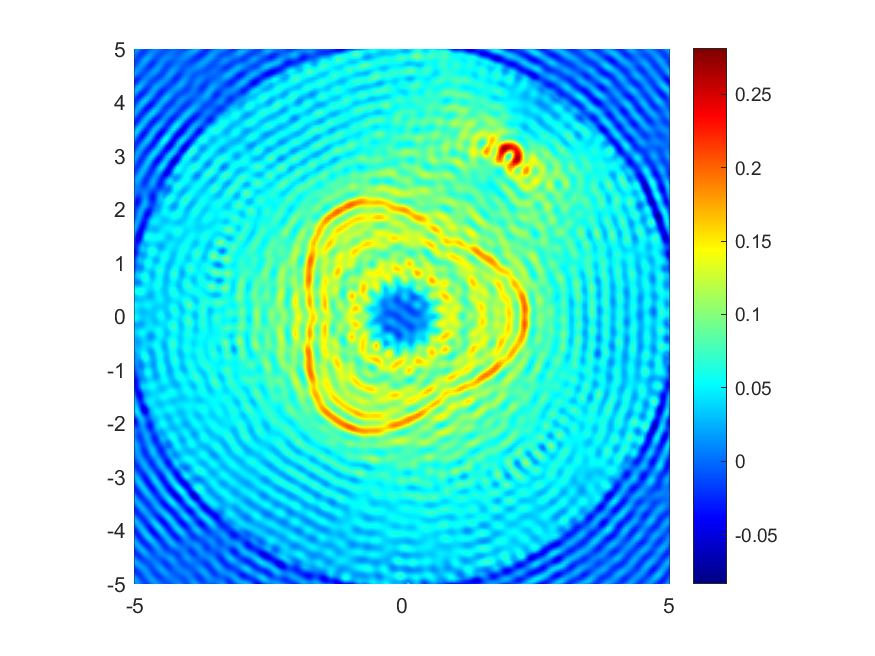}
	}
	\subfigure[$k=8\pi$]{
		\includegraphics[width=0.3\textwidth]{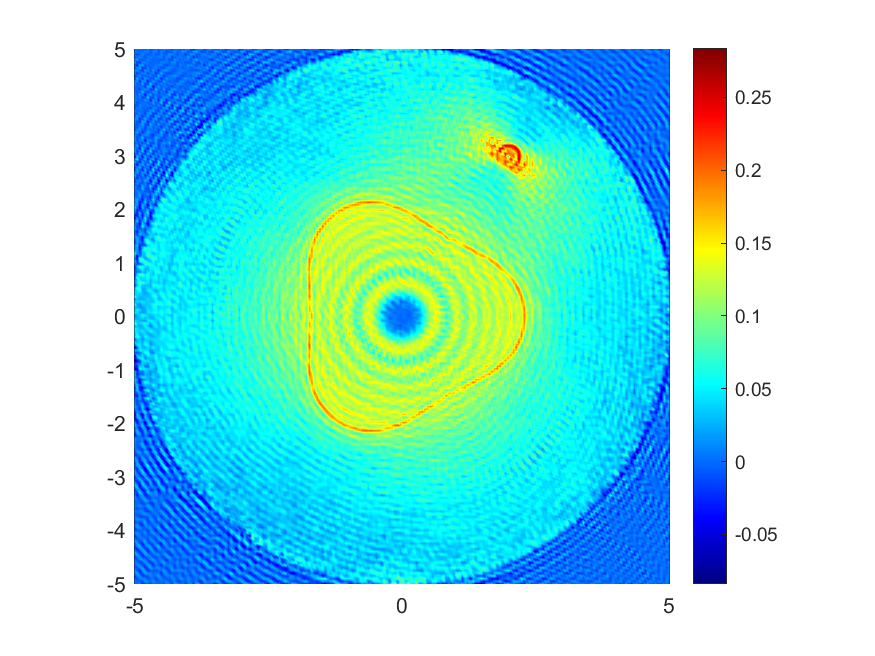}
	}
    \caption{\label{fig:ex4_multiscalar_curve}Recoveries of $\partial D$ by DCM with $k=4\pi$ and $k=8\pi$ for the multiscale case. The first column shows the exact boundary $\partial D$ and $L=256$ measurement points, the second column shows the reconstructed $\partial D$ with $k=4\pi$ and the third column shows the reconstructed $\partial D$ with $k=8\pi$.}
\end{figure}

\section{Summary}
This work focuses on developing a direct imaging method, known as the doubly cross-correlating method (DCM), for recovering a sound-soft obstacle in a passive imaging system. The fundamental challenges in this passive inverse scattering problem is the randomness of the incident sources, making direct solutions based on collected total fields inefficient. DCM addresses this by providing a cross-correlation between passive observations that can be linked to the active inverse scattering problem using the Helmholtz-Kirchhoff identity, thereby removing the randomness provided by incident point sources. This cross-correlation is then utilized to create an imaging function to visualize the obstacle. Because it relies merely on matrix multiplication, DCM is computationally efficient and simple to implement. Furthermore, the methodology presented in this work can be extended to sound-hard cases, penetrable cases, and even three-dimensional scenarios, which will be investigated in the future.

\bibliographystyle{plain}
\bibliography{DCM_passive_imaging}

\end{document}